\documentclass[12pt,a4paper]{amsart}
\usepackage{amscd,amsfonts,amsmath,amssymb,latexsym,mathrsfs,pifont,xcolor}
\numberwithin{equation}{section}\theoremstyle{plain}
\newtheorem{theorem}{Theorem}[section]
\newtheorem{proposition}[theorem]{Proposition}
\newtheorem{lemma}[theorem]{Lemma}
\newtheorem{corollary}[theorem]{Corollary}
\newtheorem{definition}[theorem]{Definition}
\newtheorem{remark}[theorem]{Remark}

\renewcommand\.{\ssf}

\renewcommand\c{\text{\rm c}}
\newcommand\C{\mathbb{C}}

\newcommand\N{\mathbb{N}}
\newcommand\R{\mathbb{R}}
\renewcommand\Re{\text{\rm Re}}
\newcommand\sign{\text{\rm sign}}
\newcommand\ssb{\hspace{-.25mm}}
\newcommand\ssf{\hspace{.25mm}}
\newcommand\supp{\operatorname{supp}}
\newcommand\Z{\mathbb{Z}}

\title[product formula for Opdam's hypergeometric functions]
{Opdam's hypergeometric functions\,:\\
product formula and convolution structure\\
in dimension 1}

\author[J.-Ph. Anker, F. Ayadi and M. Sifi]
{J.-Ph. Anker, F. Ayadi and M. Sifi}

\address{Jean-Philippe Anker, Universit\'e d'Orl\'eans \& CNRS,
F\'ed\'eration Denis Poisson (FR 2964), Laboratoire MAPMO (UMR
6628), B\^atiment de Math\'ematiques, B.P. 6759, 45067 Orl\'eans
cedex 2, France} \email{anker@univ-orleans.fr}

\address{Fatma Ayadi, Universit\'e de Tunis El Manar,
Facult\'e des Sciences de Tunis, D\'epartement de Math\'ematiques,
2092 Tunis El Manar, Tunisie} \email{ayfatm@yahoo.fr}

\address{Fatma Ayadi, Universit\'e d'Orl\'eans \& CNRS,
F\'ed\'eration Denis Poisson (FR 2964), Laboratoire MAPMO (UMR
6628), B\^atiment de Math\'ematiques, B.P. 6759, 45067 Orl\'eans
cedex 2, France} \email{fatma.ben-said@etu.univ-orleans.fr}

\address{Mohamed Sifi, Universit\'e de Tunis El Manar,
Facult\'e des Sciences de Tunis, D\'epartement de Math\'ematiques,
2092 Tunis El Manar, Tunisie} \email{mohamed.sifi@fst.run.tn}



\thanks{\textit{Acknowledgements}.
The authors were partially supported by the DGRST project
04/UR/15-02, the cooperation programs PHC Utique / CMCU 07G 1501 \&
10G 1503 and by the \textit{F\'ed\'eration Denis Poisson\/} (FR
2964).  They are grateful to the referee for his thorough checking
and detailed comments. }

\begin{document}

\begin{abstract}
Let ${\rm G}_\lambda^{(\alpha,\beta)}$ be the eigenfunctions of the
Dunkl--Cherednik operator ${\rm T}^{(\alpha,\beta)}$ on $\R$. In
this paper we express the product ${\rm
G}_\lambda^{(\alpha,\beta)}(x)\,{\rm G}_\lambda^{(\alpha,\beta)}(y)$
as an integral in terms of ${\rm G}_\lambda^{(\alpha,\beta)}(z)$
with an explicit kernel. In general this kernel is not positive.
Furthermore, by taking the so--called rational limit, we recover the
product formula of M. R\"osler for the Dunkl kernel. We then define
and study a convolution structure associated to ${\rm
G}_\lambda^{(\alpha,\beta)}.$
\end{abstract}

\maketitle

\section{Introduction}

The Opdam hypergeometric functions \,${\rm
G}_\lambda^{(\alpha,\beta)}$ on \,$\R$ \ssf are normalized
eigenfunctions
\begin{equation}\label{eigen}\begin{cases}
\;{\rm T}^{(\alpha,\beta)}\.{\rm G}_\lambda^{(\alpha,\beta)}(x)
=i\ssf\lambda\,{\rm G}_\lambda^{(\alpha,\beta)}(x)\\
\;{\rm G}_\lambda^{(\alpha,\beta)}(0)=1\,
\end{cases}\end{equation}
of the differential--difference operator
\begin{equation*}\textstyle
{\rm T}^{(\alpha,\beta)}f(x)=f^{\ssf\prime}(x)
+\underbrace{\textstyle \bigl\{(2\ssf\alpha\!+\!1)\coth
x+(2\ssf\beta\!+\!1)\tanh x\ssf\bigr\}\, \frac{f(x)\,-\,f(-x)}2
}_{\{\ssf(\alpha-\beta)\coth x\,+\,(2\ssf\beta\ssf+1)\coth2\ssf
x\ssf\}\, \{\ssf f(x)\ssf-\ssf f(-x)\ssf\}} \,-\;\rho\,f(-x)\,.
\end{equation*}
Here $\alpha\!\ge\!\beta\!\ge\!-\frac12$, $\alpha\!>\!-\frac12$,
$\rho\!=\!\alpha\!+\!\beta\!+\!1$ and $\lambda\!\in\!\C$. Notice
that, in Cherednik's notation, ${\rm T}^{(\alpha,\beta)}$ writes
\begin{equation*}\textstyle
{\rm T}(k_1,k_2) f(x)=f^{\ssf\prime}(x)
+\bigl\{\frac{2\,k_1}{1\ssf-\ssf e^{-2x}}+\frac{4\,k_2}{1\ssf-\ssf
e^{-4x}}\bigr\}\, \bigl\{\ssf f(x)\ssb-\ssb
f(-x)\ssf\bigr\}-(k_1\!+\!2\ssf k_2)\ssf f(x)\,,
\end{equation*}
with $\alpha\!=\!k_1\!+\!k_2\!-\!\frac12$ and
$\beta\!=\!k_2\!-\!\frac12$. We use as main references the article
\cite{Op} and the lecture notes \cite{O} by Opdam.

The functions ${\rm G}_\lambda^{(\alpha,\beta)}$ are closely related
to Jacobi or hypergeometric functions (see e.g. \cite[p.\;90]{Op},
\cite[Example\;7.8]{O}, \cite[Proposition\;2.1]{GT}). Specifically,
\begin{equation}\begin{aligned}\label{Gfunction}
{\rm G}_\lambda^{(\alpha,\beta)}(x)
&=\varphi^{(\alpha,\beta)}_{\lambda}(x) -\frac1{\rho\!-\!i\lambda}
\frac{\partial}{\partial x}
\varphi^{(\alpha,\beta)}_{\lambda}(x)\\
&=\varphi^{(\alpha,\beta)}_{\lambda}(x)
+\frac{\rho\!+\!i\lambda}{4(\alpha\!+\!1)}\sinh2x\,
\varphi^{(\alpha+1,\beta+1)}_{\lambda}(x)\,,
\end{aligned}\end{equation}
where \,$\varphi^{(\alpha,\beta)}_{\lambda}(x)= {}_2{\rm
F}_{\!1}\bigl(\frac{\rho+i\lambda}2,\frac{\rho-i\lambda}2\,;
\alpha\!+\!1\,;-\sinh^2\!x\bigr)$\,.

This paper deals with harmonic analysis for the functions ${\rm
G}_\lambda^{(\alpha,\beta)}$. We derive mainly a product formula for
${\rm G}_\lambda^{(\alpha,\beta)},$ which is analogous to the
corresponding result of  Flensted-Jensen and Koornwinder \cite{FK1}
for Jacobi functions, and of Ben Salem and Ould Ahmed Salem
\cite{BO} for Jacobi--Dunkl functions. The product formula is the
key information needed in order to define an associated convolution
structure on $\R.$ More precisely, we deduce the product formula
\begin{equation}\label{pint}
{\rm G}_\lambda^{(\alpha,\beta)}(x)\,{\rm
G}_\lambda^{(\alpha,\beta)}(y) =\int_{\R}{\rm
G}_\lambda^{(\alpha,\beta)}(z)\,d\mu_{\,x,y}^{(\alpha,\beta)}(z)
\quad\forall x,y\in\R, \quad\forall\lambda\in \C,
\end{equation}
from the corresponding formula for
$\varphi^{(\alpha,\beta)}_\lambda$ on $\R^+$. Here
$\mu_{\,x,y}^{(\alpha,\beta)}$ is an explicit real valued measure
with compact support on $\R$, which may not be positive and which is
uniformly bounded in $x,y\in \R$. We conclude the first part of the
paper by recovering as a limit case the product formula for the
Dunkl kernel obtained in \cite{Roes}.

In the second part of the paper, we use the product formula
\eqref{pint} to define and study the translation operators
$$
\tau_{\,x}^{(\alpha,\beta)}f(y)
:=\int_{\R}f(z)\,d\mu_{\,x,y}^{(\alpha,\beta)}(z).
$$
We next define the convolution product of suitable functions $f$ and
$g$ by
$$
f\ast_{\alpha,\beta} g(x)
:=\int_{\R}\tau_{\,x}^{(\alpha,\beta)}f(-y)\,g(y)\,A_{\alpha,\beta}(|y|)\,dy,
$$
where $A_{\alpha,\beta}(y)\!=\!(\sinh y)^{2\alpha+1}(\cosh
y)^{2\beta+1}$. We show in particular that
$f\ast_{\alpha,\beta}g=g\ast_{\alpha,\beta}f$ and that $\mathcal
F(f\ast_{\alpha,\beta}g)=\mathcal F(f)\,\mathcal F(g)$, where
$\mathcal F$ is the so--called Opdam--Cherednik transform.
Eventually we prove an analog of the Kunze--Stein phenomenon for the
$\ast_{\alpha,\beta}$-convolution product of $L^p$-spaces.

In the last part of the paper, we construct an orthogonal basis of
the Hilbert space $L^2(\R,A_{\alpha,\beta}(|x|)\ssf dx)$,
generalizing the corresponding result of Koornwinder \cite{K2} for
$L^2(\R^+\!, A_{\alpha,\beta}(x)\ssf dx)$. As a limit case, we
recover the Hermite functions constructed by Rosenblum \cite{Rosn}
in $L^2(\R, \vert x\vert^{2\alpha+1}dx)$.

Our paper is organised as follows. In section 2, we recall some
properties and formulas for Jacobi functions. In section 3, we give
the proof of the product formula for ${\rm
G}_\lambda^{(\alpha,\beta)}.$ Section 4 is devoted to the
translation operators and the associated convolution product.
Section 5 contains  a Kunze--Stein type phenomenon. In Section 6, we
construct an orthogonal basis of $L^2(\R, A_{\alpha,\beta}(|x|)\ssf
dx)$ and compute its Opdam--Cherednik transform.

\section{Preliminaries}\label{recalls}

In this section we recall some properties of the Jacobi functions.
See \cite{FK1} and  \cite{FK2} for more details, as well as the
survey \cite{K2}.

Let $\alpha\!\ge\!\beta\!\ge\!-\frac12$ with $\alpha\!\ne\!-\frac12$
and $\lambda\in \mathbb C$. The Jacobi function
$\varphi_{\lambda}^{(\alpha,\beta)}$ is defined by
\begin{align}\label{Jacobi}
\varphi_{\lambda}^{(\alpha,\beta)}(x)&={}_2{\rm F}_{\!1}\Bigl(
\frac{\rho\!+\!i\lambda}2,\frac{\rho\!-\!i\lambda}2;
\alpha\!+\!1\ssf;-\sinh^2\!x\Bigr)\\
&=(\cosh x)^{-\rho-i\lambda}\,{}_2{\rm F}_{\!1}\Bigl(
\frac{\rho\!+\!i\lambda}2,\frac{\alpha\!-\!\beta\!+\!1\!+i\lambda}2;
\alpha\!+\!1\ssf;\tanh^2\!x\Bigr)
\quad\forall\;x\!\in\!\R\,,\nonumber\end{align} where
$\rho=\alpha\!+\!\beta\!+\!1$ and ${}_2{\rm F}_{\!1}$ denotes the
hypergeometric function.

Its asymptotic behavior is generically given by
\begin{equation}\label{JacobiAsymptotic}
\varphi_{\lambda}^{(\alpha,\beta)}(x)
=\c_{\alpha,\beta}(\lambda)\,\Phi_{\lambda}^{(\alpha,\beta)}(x)
+\c_{\alpha,\beta}(-\lambda)\,\Phi_{-\lambda}^{(\alpha,\beta)}(x)
\qquad\forall\,\lambda\!\in\!\C\!\smallsetminus\!i\Z\,,\,\forall\,x\!\in\!\R^*,
\end{equation}
where
\begin{equation}\label{cfunction}
\c_{\alpha,\beta}(\lambda)
=\frac{\Gamma(2\alpha\!+\!1)}{\Gamma(\alpha\!+\!\frac12)}
\frac{\Gamma(i\lambda)}{\Gamma(\alpha\!-\!\beta\!+\!i\lambda)}
\frac{\Gamma(\frac{\alpha-\beta+i\lambda}2)}{\Gamma(\frac{\rho+i\lambda}2)}
=\frac{\Gamma(\alpha\!+\!1)\,2^{\ssf\rho-i\lambda}\,\Gamma(i\lambda)}
{\Gamma(\frac{\rho+i\lambda}2)\,\Gamma(\frac{\alpha-\beta+1+i\lambda}2)}
\end{equation}
and
\begin{equation}\label{Phifunction}
\Phi_{\lambda}^{(\alpha,\beta)}(x) =(2\cosh
x)^{-\rho+i\lambda}\,{}_2{\rm F}_{\!1}\Bigl(
\frac{\rho\!-\!i\lambda}2,\frac{\alpha\!-\!\beta\!+\!1\!-\!i\lambda}2\ssf;
1\!-\!i\lambda\ssf;\cosh^{-2}\!x\Bigr).
\end{equation}
In the limit case $\lambda\!=\!0$, we obtain
\begin{equation}\label{Jacobi0Asymptotic}
\varphi_0^{(\alpha,\beta)}(x)
=\frac{2^{\ssf\rho+1}\,\Gamma(\alpha\!+\!1)}
{\Gamma(\frac\rho2)\,\Gamma(\frac{\alpha-\beta+1}2)}\,
|x|\,e^{-\rho\,|x|} +\mathcal{O}\bigl(e^{-\rho\,|x|}\bigr)
\qquad\text{as \;}|x|\!\to\!+\infty\,,
\end{equation}
after multiplying \eqref{JacobiAsymptotic} by $\lambda$ and applying
$\frac\partial{\partial\lambda}\big|_{\lambda=0}$.

The Jacobi functions satisfy the following product formula, for
$\alpha\!>\!\beta>\!-\frac12$  and $x,y\!\ge\!0$\,:
\begin{equation}\label{formula prod1}
\varphi_{\lambda}^{(\alpha,\beta)}(x)\,\varphi_{\lambda}^{(\alpha,\beta)}(y)
=\int_{0}^{1}\int_{0}^{\pi}
\varphi_{\lambda}^{(\alpha,\beta)}(\arg\cosh|\gamma(x,y,r,\psi)|)\,
dm_{\alpha,\beta}(r,\psi),
\end{equation}
where
$$
\gamma(x,y,r,\psi)=\cosh x\cosh y+\sinh x\sinh y\,re^{i\psi},
$$
and
\begin{equation}\label{measure}
dm_{\alpha,\beta}(r,\psi)
=2\,M_{\alpha,\beta}\,(1-r^2)^{\alpha-\beta-1}(r\sin\psi)^{2\beta}\,r\,dr\,d\psi
\end{equation}
with
$$
M_{\alpha,\beta}=\frac{\Gamma(\alpha\!+\!1)}
{\sqrt{\pi}\,\Gamma(\alpha\!-\!\beta)\,\Gamma(\beta\!+\!\frac12)}.
$$
When \ssf$\alpha\ssb=\ssb\beta\ssb>\ssb-\ssb\frac12$, the product
formula becomes
\begin{equation}
\label{formula prod11}
\varphi_{\lambda}^{(\alpha,\alpha)}(x)\,\varphi_{\lambda}^{(\alpha,\alpha)}(y)
=M_{\alpha,\alpha}\int_{0}^{\pi}
\varphi_{\lambda}^{(\alpha,\alpha)}(\arg\cosh|\gamma(x,y,1,\psi)|)\,
(\sin\psi)^{2\alpha}\,d\psi,
\end{equation}
where $M_{\alpha,\alpha}=\frac{\Gamma(\alpha+1)}
{\sqrt{\pi}\,\Gamma(\alpha+\frac12)}$. Notice that the limit cases
\ssf$\alpha>\beta=-\frac12$ \ssf and
\ssf$\alpha\ssb=\ssb\beta\ssb>\ssb-\frac12$ \ssf are connected by
the quadratic transformation
\ssf$\varphi_\lambda^{(\alpha,-\frac12)}(x)
=\varphi_{2\lambda}^{(\alpha, \alpha)}\bigl(\frac x2\bigr)$.

For $\alpha>\beta>-\frac 12$ and fixed $x,y>0$, we perform the
change of variables \,$]\ssf0,1\ssf[\,\times\,
]\ssf0,\pi\ssf[\hspace{1mm}\ni\hspace{-.25mm}
(r,\psi)\ssf\longmapsto\ssf(z,\chi)
\hspace{-.25mm}\in\hspace{1mm}]\ssf0,+\infty\ssf[\,
\times\,]\ssf0,\pi\ssf[$ \,defined by
\begin{equation}\label{ChangeVariables}
\cosh z\ e^{i\chi}=\gamma(x,y,r,\psi)
\quad\Longleftrightarrow\quad\begin{cases}
\;r\cos\psi=\frac{\cosh z\cos\chi\,-\,\cosh x\cosh y}{\sinh x\,\sinh y}\,,\\
\;r\sin \psi=\frac{\cosh z\,\sin\chi}{\sinh x\,\sinh y}\,.
\end{cases}\end{equation}
This implies in particular that
$$
\cosh (x\!-\!y)\leq\cosh(z)\leq\cosh(x\!+\!y),
$$
and therefore $x$, $y$, $z$ satisfy the triangular inequality
$$
|x\!-\!y|\leq|z|\le x\!+\!y.
$$
Moreover, an easy computation gives
$$
1-r^2= (\sinh x \sinh  y)^{-2}\,g(x,y,z,\chi),
$$
where
\begin{equation}\label{g-fun}
g(x,y,z,\chi):=1-\cosh^2\!x-\cosh^2\!y-\cosh^2\!z+2\cosh x\cosh
y\cosh z\cos\chi.
\end{equation}
Furthermore, the measure \,$\sinh^2\!x\sinh^2\!y\,r\,drd\psi$
\,becomes \,$\cosh z\sinh z\,dz\,d\chi$ \,and therefore the measure
\eqref{measure} becomes
$$
dm_{\alpha,\beta}(r,\psi)\hspace{-.25mm}= 2\ssf
M_{\alpha,\beta}\hspace{.75mm} g(x,y,z,\chi)^{\alpha-\beta-1}
\bigl(\sinh x\sinh y\sinh z\bigr)^{-2\alpha}
(\sin\chi)^{2\beta}A_{\alpha,\beta}(z)\,dz\ssf d\chi,
$$
where
\begin{equation}\label{A-fun}
A_{\alpha,\beta}(z):=(\sinh z)^{2\alpha+1}(\cosh z)^{2\beta+1}.
\end{equation}
Hence,  the product formula \eqref{formula prod1} reads
\begin{equation}
\label{formula prod2}
\varphi_{\lambda}^{(\alpha,\beta)}(x)\,\varphi_{\lambda}^{(\alpha,\beta)}(y)
=\int_{\,0}^{+\infty}\!\varphi_{\lambda}^{(\alpha,\beta)}(z)\,
W_{\alpha,\beta}(x,y,z)\,A_{\alpha,\beta}(z)\,dz,\qquad x,y>0,
\end{equation}
where
\begin{equation*}
W_{\alpha,\beta}(x,y,z) :=2\,M_{\alpha,\beta}\,(\sinh x\sinh y\sinh
z)^{-2\alpha}\!
\int_{0}^{\pi}\!g(x,y,z,\chi)_{+}^{\alpha-\beta-1}(\sin\chi)^{2\beta}\,d\chi
\end{equation*}
if $x,y,z>0$ satisfy $|x-y|<z<x+y$ and $W_{\alpha,\beta}(x,y,z)=0$
otherwise. Here
$$
g_{+}=\begin{cases}
\;g&\text{if \;}g>0,\\
\;0&\text{if \;}g\le0.
\end{cases}
$$
We point out that the function  $W_{\alpha,\beta}(x,y,z)$ is
nonnegative, symmetric in the variables $x,y,z$ and that
$$
\int_{\,0}^{+\infty}\!W_{\alpha,\beta}(x,y,z)\,A_{\alpha,\beta}(z)\,dz=1.
$$
Furthermore, in \cite[Formula (4.19)]{FK1} the authors express
$W_{\alpha,\beta}$ as follows in terms of the hypergeometric
function  ${}_2{\rm F}_{\!1}$ : For every $x,y,z>0$ satisfying the
triangular inequality $|x-y|<z<x+y,$
\begin{align}\label{W}
W_{\alpha,\beta}(x,y,z) &=M_{\alpha,\alpha}\, (\cosh x\cosh y\cosh
z)^{\alpha-\beta-1}\,
(\sinh x\sinh y\sinh z)^{-2\alpha}\nonumber\\
&\times(1-B^2)^{\alpha-\frac12}\, {}_2{\rm
F}_{\!1}\Bigl(\alpha+\beta,\alpha-\beta;\alpha+\frac{1}{2};\frac{1-B}{2}\Bigr),
\end{align}
where
$$
B:=\frac{\cosh^2\!x+\cosh^2\!y+\cosh^2\!z-1}{2\cosh x\cosh y\cosh
z}.
$$
Notice that
$$
1\pm B=\frac {[\,\cosh(x\!+\!y)\pm\cosh z\,]\,[\,\cosh z\pm\cosh
(x\!-\!y)\,]} {2\cosh x\cosh y\cosh z},
$$
hence
\begin{align}\label{1-B2}
1-B^2 &=\frac
{[\,\cosh2(x\!+\!y)-\cosh2z\,]\,[\,\cosh2z-\cosh2(x\!-\!y)\,]}
{16\cosh^2\!x\cosh^2\!y\cosh^2\!z} \\
&=\frac
{\sinh(x\!+\!y\!+\!z)\sinh(\!-x\!+\!y\!+\!z)\sinh(x\!-\!y\!+\!z)\sinh(x\!+\!y\!-\!z)}
{4\cosh^2\!x\cosh^2\!y\cosh^2\!z}.\nonumber
\end{align}

In the case $\alpha=\beta>-\frac12$, we use instead the change of
variables
$$
\cosh z=|\gamma(x,y,1,\psi)|=|\cosh x\cosh y+\sinh x\sinh
y\,e^{i\psi}\,|\,,
$$
and we obtain the same product formula \eqref{formula prod2}, where
$W_{\alpha, \alpha}$ is given by
\begin{equation}\begin{aligned}\label{Walpha}
&W_{\alpha,\alpha}(x,y,z) =2^{\ssf4\alpha+1}M_{\alpha,\alpha}\,
\bigl[\sinh 2x\sinh 2y\sinh 2z\bigr]^{-2\alpha}\\
&\times\bigl[\sinh(x\!+\!y\!+\!z)\sinh(-x\!+\!y\!+\!z)
\sinh(x\!-\!y\!+\!z)\sinh(x\!+\!y\!-\!z)\bigr]^{\alpha-1/2}.
\end{aligned}\end{equation}
In the case $\alpha>\beta=-\frac12$, we use the quadratic
transformation
$$
\varphi_{\lambda}^{(\alpha,-\frac{1}{2})}(2x)
=\varphi_{2\lambda}^{(\alpha,\alpha)}(x),
$$
and we obtain again the product formula \eqref{formula prod2}, with
$$
W_{\alpha,-\frac{1}{2}}(x,y,z)
=2^{-2\alpha}\,W_{\alpha,\alpha}(\frac{x}{2},\frac{y}{2},\frac{z}{2}).
$$

As noticed by Koornwinder \cite{K1} (see also \cite{FK2}), the
product formulas \eqref{formula prod1} and \eqref{formula prod2} are
closely connected with the addition formula for the Jacobi
functions, that we recall now for later use\,:
\begin{equation}\begin{aligned}\label{AdditionFormula}
&\varphi_\lambda^{(\alpha,\beta)}(\arg\cosh|\gamma(x,y,r,\psi)|)\\
&=\sum\nolimits_{\,0\le\ell\le k<\infty}
\varphi_{\lambda,k,\ell}^{(\alpha,\beta)}(x)\,
\varphi_{-\lambda,k,\ell}^{(\alpha,\beta)}(-y)\,
\chi_{\,k,\ell}^{(\alpha,\beta)}(r,\psi)\,
\Pi_{\,k,\ell}^{(\alpha,\beta)}\,,
\end{aligned}\end{equation}
where
\begin{equation*}
\varphi_{\lambda,k,\ell}^{(\alpha,\beta)}(x)
=\frac{\c_{\alpha,\beta}(-\lambda)}
{\c_{\alpha+k+\ell,\beta+k-\ell}(-\lambda)}\, (2\sinh
x)^{k-\ell}\,(2\cosh x)^{k+\ell}\,
\varphi_\lambda^{(\alpha+k+\ell,\beta+k-\ell)}(x)
\end{equation*}
are modified Jacobi functions, the functions
\begin{equation*}\textstyle
\chi_{\,k,\ell}^{(\alpha,\beta)}(r,\psi) =r^{k-\ell}\,
\frac{\ell\,!}{(\alpha-\beta)_\ell}\, {\rm
P}_{\ssf\ell}^{(\alpha-\beta-1,\ssf\beta+k-\ell)\vphantom{\frac12}}(2r^2\hspace{-1mm}-\!1)\,
\frac{(k-\ell\ssf)!}{(\beta\ssf+\frac12)_{k-\ell}}\, {\rm P}_{\ssf
k-\ell}^{(\beta\ssf-\frac12,\beta\ssf-\frac12)}(\cos\psi)\,,
\end{equation*}
which are expressed in terms of Jacobi polynomials (see for instance
\cite{AAR})
\begin{equation}\label{JacobiPolynomial}\textstyle
{\rm P}_{\ssf n}^{(a,b)}(z) =\frac{(a\ssf+1)_n}{n\ssf!}\,{}_2{\rm
F}_{\!1} \bigl(\ssb-\ssf
n\ssf,a\!+\!b\!+\!n\!+\!1\ssf;a\!+\!1\ssf;\frac{1-z}2\ssf\bigr)\,,
\end{equation}
are orthogonal with respect to the measure \eqref{measure}, and
\begin{equation}\begin{aligned}\label{pi}
\Pi_{\,k,\ell}^{(\alpha,\beta)} &=\Big(\int_0^1\int_0^\pi
\chi_{k,\ell}^{(\alpha,\beta)}(r,\psi)^2\,
dm_{\alpha,\beta}(r,\psi)\Big)^{-1}\\
&\textstyle\;
=\frac{(\alpha+k+\ell)\,(\beta+2k-2\ell)}{(\alpha+k)\,(2\beta+k-\ell)}\,
\frac{(\alpha+1)_k\,(\alpha-\beta)_\ell\,(2\beta+1)_{k-\ell}}
{(\beta+1)_k\,\ell\ssf!\,(k-\ell)!}\,.
\end{aligned}\end{equation}

\section{Product formula for ${\rm G}_\lambda^{(\alpha,\beta)}$}

For $x,y,z\in \R$ and $\chi\in [0,\pi],$ let
\begin{equation}\label{sigma-fun}
\sigma_{x,y,z}^{\chi}=\begin{cases} \,\frac{\cosh x\,\cosh
y\,-\,\cosh z\,\cos\chi}{\sinh x\,\sinh y}
&\text{if \;}xy\neq0\,,\\
\qquad0 &\text{if \;}xy=0\,.
\end{cases}\end{equation}
Furthermore, if $\alpha>\beta>-\frac12,$ let us define $\mathcal
K_{\alpha,\beta}$ by
\begin{align*}
\mathcal{K}_{\alpha,\beta}(x,y,z) =&\hspace{1mm}M_{\alpha,\beta}\,
\bigl|\,\sinh x\,\sinh y\,\sinh z\,\bigr|^{-2\alpha}
\int_{0}^{\pi}g(x,y,z,\chi)_{+}^{\alpha-\beta-1}\\
&\times
\Bigl[1-\sigma_{x,y,z}^{\chi}+\sigma_{x,z,y}^{\chi}+\sigma_{z,y,x}^{\chi}
+\frac{\rho}{\beta+\frac{1}{2}}\coth x\coth y\coth z\,(\sin\chi)^2\Bigr]\\
&\hspace{90mm}\times(\sin\chi)^{2\beta}\,d\chi
\end{align*}
if $x,y,z\in\R^*$ satisfy the triangular inequality
$||x|-|y||<|z|<|x|+|y|$, and $\mathcal{K}_{\alpha,\beta}(x,y,z)=0$
otherwise. Here $g(x,y,z,\chi)$ is as in \eqref{g-fun}.
\begin{remark}\label{symmetry}
The following symmetry properties are easy to check:
$$
\begin{cases}
\;\mathcal{K}_{\alpha,\beta}(x,y,z)=\mathcal{K}_{\alpha,\beta}(y,x,z),\\
\;\mathcal{K}_{\alpha,\beta}(x,y,z)=\mathcal{K}_{\alpha,\beta}(-z,y,-x),\\
\;\mathcal{K}_{\alpha,\beta}(x,y,z)=\mathcal{K}_{\alpha,\beta}(x,-z,-y).
\end{cases}
$$
\end{remark}

Recall the Opdam functions ${\rm G}_\lambda^{(\alpha,\beta)}$
defined in \eqref{Gfunction}. This section is devoted to the proof
of our main result, that we state first in the case
$\alpha\!>\!\beta\!>\!-\frac12$.

\begin{theorem}\label{productalpha>beta}
Assume \,$\alpha\!>\!\beta\!>\!-\frac12$. Then \,${\rm
G}_\lambda^{(\alpha,\beta)}$ satisfies the following product formula
\begin{equation*}\label{PROD}
{\rm G}_\lambda^{(\alpha,\beta)}(x)\,{\rm
G}_\lambda^{(\alpha,\beta)}(y) =\int_{-\infty}^{+\infty}\! {\rm
G}_\lambda^{(\alpha,\beta)}(z)\,d\mu_{\,x,y}^{(\alpha,\beta)}(z),
\end{equation*}
for $x,y\in\R$ and $\lambda\in\C.$ Here
\begin{equation}\label{measure1}
d\mu_{\,x,y}^{(\alpha,\beta)}(z)=
\begin{cases}
\,\mathcal{K}_{\alpha,\beta}(x,y,z)\,A_{\alpha,\beta}(\vert
z\vert)\,dz
&\text{if \;}xy\neq0\\
\qquad d\delta_x(z)
&\text{if \;}y=0\\
\qquad d\delta_y(z) &\text{if \;}x=0
\end{cases}\end{equation}
and $A_{\alpha,\beta}$ is as in \eqref{A-fun}.
\end{theorem}
\vspace{2mm}

Let us split the Opdam function
$$
{\rm G}_\lambda^{(\alpha,\beta)} ={\rm G}_{\lambda,\text{\rm
e}}^{(\alpha,\beta)}+{\rm G}_{\lambda,\text{\rm o}}^{(\alpha,\beta)}
$$
into its even part
$$
{\rm G}_{\lambda,\text{\rm
e}}^{(\alpha,\beta)}(x)=\varphi^{(\alpha,\beta)}_{\lambda}(x)
$$
and odd part
$$
{\rm G}_{\lambda,\text{\rm o}}^{(\alpha,\beta)}(x)
=-\,\frac{1}{\rho\!-\!i\lambda}\,\frac{\partial}{\partial x}\,
\varphi^{(\alpha,\beta)}_{\lambda}(x)
=\frac{\rho\!+\!i\lambda}{4(\alpha\!+\!1)}\sinh2x\,
\varphi^{(\alpha+1,\beta+1)}_{\lambda}(x).
$$

For $x,y\in \R^*$, the product formula \eqref{formula prod2} for the
Jacobi functions yields
\begin{align*}
{\rm G}_{\lambda,\text{\rm e}}^{(\alpha,\beta)}(x)\,{\rm
G}_{\lambda,\text{\rm e}}^{(\alpha,\beta)}(y) &=\int_{\vert \vert
x\vert-\vert y\vert\vert}^{\vert x\vert +\vert y\vert} {\rm
G}_{\lambda,\text{\rm e}}^{(\alpha,\beta)}(z)\,
W_{\alpha,\beta}(|x|,|y|,z)\,A_{\alpha,\beta}(z)\,dz\\
&=\frac12\int_{I_{x,y}}{\rm G}_\lambda^{(\alpha,\beta)}(z)\,
W_{\alpha,\beta}(|x|,|y|,|z|)\,A_{\alpha,\beta}(\vert z\vert)\,dz,
\end{align*}
where \begin{equation}\label{I}
I_{x,y}:=[-|x|\!-\!|y|,-||x|\!-\!|y||]\cup[||x|\!-\!|y||,|x|\!+\!|y|]\,.
\end{equation}

Next let us turn to the mixed products. The following statement
amounts to Lemma 2.3 in \cite{BO}.

\begin{lemma}\label{oddeven}
For \,$\alpha\!>\!\beta\!>\!-\frac12$, $\lambda\!\in\!\C$ and
\,$x,y\!\in\!\R^*$, we have
\begin{align*}
{\rm G}_{\lambda,\text{\rm o}}^{(\alpha,\beta)}(x)\,{\rm
G}_{\lambda,\text{\rm e}}^{(\alpha,\beta)}(y)
&=M_{\alpha,\beta}\int_{I_{x,y}}{\rm
G}_\lambda^{(\alpha,\beta)}(z)\,
\bigl|\,\sinh x\,\sinh y\,\sinh z\,\bigr|^{-2\alpha}\\
&\times\Big\{\int_{0}^{\pi}
g(x,y,z,\chi)_{+}^{\alpha-\beta-1}\,\sigma_{x,z,y}^{\chi}\,
(\sin\chi)^{2\beta}\,d\chi\Big\}\,A_{\alpha,\beta}(|z|)\,dz\,,
\end{align*}
where  $g(x,y,z,\chi)$ is given by \eqref{g-fun},
$\sigma_{x,z,y}^{\chi}$ by \eqref{sigma-fun} and $I_{x,y}$ by
\eqref{I}.
\end{lemma}

We consider now purely odd products, which is the most difficult
case.

\begin{lemma}\label{oddodd}
For \,$\alpha\!>\!\beta\!>\!-\frac12$, $\lambda\!\in\!\C$ and
\,$x,y\!\in\!\R^*$, we have
\begin{align*}
{\rm G}_{\lambda,\text{\rm o}}^{(\alpha,\beta)}(x)\,&{\rm
G}_{\lambda,\text{\rm o}}^{(\alpha,\beta)}(y)
=M_{\alpha,\beta}\int_{I_{x,y}} {\rm
G}_\lambda^{(\alpha,\beta)}(z)\,
\bigl|\,\sinh x\,\sinh y\,\sinh z\,\bigr|^{-2\alpha}\\
&\times\Big\{\int_{0}^{\pi}g(x,y,z,\chi)_{+}^{\alpha-\beta-1}
\Bigl[-\,\sigma_{x,y,z}^{\chi}
-\frac{\rho}{(\beta+\frac12)}\coth x\coth y\coth z\,(\sin\chi)^2\Bigr]\\
&\hspace{77mm}\times(\sin\chi)^{2\beta}d\chi\Big\}\,A_{\alpha,\beta}(|z|)\,dz.
\end{align*}
\end{lemma}

\begin{proof}
For $x,y\!>\!0$, we have
\begin{align}\label{I1+I2}
{\rm G}_{\lambda,\text{\rm o}}^{(\alpha,\beta)}(x)\,{\rm
G}_{\lambda,\text{\rm o}}^{(\alpha,\beta)}(y)
&=\frac{(\rho\!+\!i\lambda)^2}{16\,(\alpha\!+\!1)^2}\,
\sinh2x\,\sinh2y\; \varphi_{\lambda}^{(\alpha+1,\beta+1)}(x)\,
\varphi_{\lambda}^{(\alpha+1,\beta+1)}(y)\nonumber\\
&=\,\mathcal{I}_{\lambda,1}^{(\alpha,\beta)}(x,y)
+\mathcal{I}_{\lambda,2}^{(\alpha,\beta)}(x,y),
\end{align}
where
\begin{align*}
\mathcal{I}_{\lambda,1}^{(\alpha,\beta)}(x,y)
:=-\,\frac{\rho^{\hspace{.5mm}2}\hspace{-1mm}
+\!\lambda^2}{16\,(\alpha\!+\!1)^2}\, \sinh2x\,\sinh2y\;
\varphi_{\lambda}^{(\alpha+1,\beta+1)}(x)\,
\varphi_{\lambda}^{(\alpha+1,\beta+1)}(y),
\end{align*}
and
\begin{align}\label{I-2}
\mathcal{I}_{\lambda,2}^{(\alpha,\beta)}(x,y)
:=\frac{\rho\,(\rho\!+\!i\lambda)}{8\,(\alpha\!+\!1)^2}\,
\sinh2x\,\sinh2y\; \varphi_{\lambda}^{(\alpha+1,\beta+1)}(x)
\,\varphi_{\lambda}^{(\alpha+1,\beta+1)}(y).
\end{align}

Consider first \,$\mathcal{I}_{\lambda,1}^{(\alpha,\beta)}$. We
deduce from the addition formula \eqref{AdditionFormula} that
\begin{equation*}
\int_{0}^{1}\!\int_{0}^{\pi}
\varphi_{\lambda}^{(\alpha,\beta)}(\arg\cosh|\gamma(x,y,r,\psi)|)\,
\chi_{\,1,0}^{(\alpha,\beta)}(r,\psi)\,dm_{\alpha,\beta}(r,\psi)
=\,\varphi_{\lambda,1,0}^{(\alpha,\beta)}(x)\,
\varphi_{-\lambda,1,0}^{(\alpha,\beta)}(-y),
\end{equation*}
where \,$\chi_{\,1,0}^{(\alpha,\beta)}(r,\psi)=r\cos\psi$ \,and
\,$\varphi_{\pm\lambda,1,0}^{(\alpha,\beta)}(x)
=\frac{\rho\ssf\mp\ssf i\lambda} {4\ssf(\alpha+1)}\sinh
2x\,\varphi_\lambda^{(\alpha+1,\beta+1)}(x)$. Hence
\begin{equation}\label{I-1}
\mathcal{I}_{\lambda,1}^{(\alpha,\beta)}(x,y)
=\int_{0}^{1}\!\int_{0}^{\pi}\varphi_{\lambda}^{(\alpha,\beta)}
(\arg\cosh|\gamma(x,y,r,\psi)|)\;r\cos\psi\,dm_{\alpha,\beta}(r,\psi).
\end{equation}
By performing the change of variables \eqref{ChangeVariables} and
arguing as in Section \ref{recalls}, \eqref{I-1} becomes
\begin{align*}
\mathcal{I}_{\lambda,1}^{(\alpha,\beta)}(x,y)
&=-\,2\,M_{\alpha,\beta}\int_{|x-y|}^{\,x+y}
{\rm G}_{\lambda,\text{\rm e}}^{(\alpha,\beta)}(z)\,(\sinh x\sinh y\sinh z)^{-2\alpha}\\
&\,\times\Big\{\int_{0}^{\pi}\sigma_{x,y,z}^\chi\,
g(x,y,z,\chi)_{+}^{\alpha-\beta-1}\,(\sin\chi)^{2\beta}\,d\chi\Big\}\,
A_{\alpha,\beta}(z)\,dz.
\end{align*}
By using the symmetries
\begin{equation*}\begin{cases}
\;g(x,y,z,\chi)\ssf
=\,g(|x|,|y|,|z|,\chi)\,,\\
\;\mathcal{I}_{\lambda,1}^{(\alpha,\beta)}(x,y)\ssf
=\,\sign(xy)\,\mathcal{I}_{\lambda,1}^{(\alpha,\beta)}(|x|,|y|)\,,\\
\;\sigma_{x,y,z}^{\chi} =\,\sign(xy)\,\sigma_{|x|,|y|,|z|}^{\chi}\,,
\end{cases}\end{equation*}
we conclude, for all $x,y\!\in\!\R^*$, that
\begin{equation*}\begin{aligned}
\mathcal{I}_{\lambda,1}^{(\alpha,\beta)}(x,y)
&=-\,M_{\alpha,\beta}\int_{I_{x,y}} {\rm
G}_\lambda^{(\alpha,\beta)}(z)\,
\bigl|\,\sinh x\,\sinh y\,\sinh z\,\bigr|^{-2\alpha}\\
&\,\times\Big\{\int_{0}^{\pi}\sigma_{x,y,z}^\chi\;
g(x,y,z,\chi)_{+}^{\alpha-\beta-1}(\sin\chi)^{2\beta}\,d\chi\Big\}\,
A_{\alpha,\beta}(\vert z\vert)\,dz.
\end{aligned}\end{equation*}

Consider next \,$\mathcal{I}_{\lambda,2}^{(\alpha,\beta)}$. By using
this time the product formula \eqref{formula prod2} for
$\varphi_\lambda^{(\alpha+1,\beta+1)}$, we obtain, for $x,y\!>\!0$,
\begin{align*}
\mathcal{I}_{\lambda,2}^{(\alpha,\beta)}(x,y)
&=\frac{\rho\,(\rho\!+\!i\lambda)}{4\,(\alpha\!+\!1)^2}\,
M_{\alpha+1,\beta+1}\sinh2x\sinh2y\\
&\ssf\times
\int_{|x-y|}^{x+y}\varphi_\lambda^{(\alpha+1,\beta+1)}(z)
(\sinh x\sinh y\sinh z)^{-2\alpha-2}\\
&\ssf\times
\Bigl\{\int_{0}^{\pi}\!g(x,y,z,\chi)_{+}^{\alpha-\beta-1}
(\sin \chi)^{2\beta+2}d\chi\Bigr\}\,A_{\alpha+1,\beta+1}(z)\,dz\\
&=2\,M_{\alpha,\beta}\int_{|x-y|}^{x+y} {\rm G}_{\lambda,\text{\rm
o}}^{(\alpha,\beta)}(z)\,
(\sinh x\sinh y\sinh z)^{-2\alpha}\frac\rho{\beta\!+\!\frac12}\\
&\ssf\times\coth x \coth y\coth z\,
\Bigl\{\int_{0}^{\pi}\!g(x,y,z,\chi)_{+}^{\alpha-\beta-1}
(\sin\chi)^{2\beta+2}d\chi\Bigr\}\,A_{\alpha,\beta}(z)\,dz\,.
\end{align*}
By arguing again by evenness and oddness, we deduce, for all
$x,y\!\in\!\R^*$,
\begin{equation*}\begin{aligned}
&\mathcal{I}_{\lambda,2}^{(\alpha,\beta)}(x,y)
=M_{\alpha,\beta}\int_{I_{x,y}}{\rm G}_\lambda^{(\alpha,\beta)}(z)\,
\bigl|\,\sinh x\,\sinh y\,\sinh z\,\bigr|^{-2\alpha}\\
&\times\frac{\rho}{\beta\!+\!\frac12}\, (\coth x\coth y\coth z)\,
\Bigl\{\int_{0}^{\pi}g(x,y,z,\chi)_{+}^{\alpha-\beta-1}
(\sin\chi)^{2\beta+2}d\chi\Bigr\}\, A_{\alpha,\beta}(|z|)\,dz\,.
\end{aligned}\end{equation*}
This concludes the proof of Lemma \ref{oddodd} and hence the proof
of Theorem \ref{productalpha>beta}.
\end{proof}

Next we turn our attention to the case $\alpha=\beta>-\frac{1}{2}$.
For $x,y,z\in \R,$ let
\begin{equation}\label{sigma2-fun}
\sigma_{x,y,z}=\begin{cases} \frac{\cosh 2x\,\cosh 2y\,-\,\cosh
2z}{\sinh 2x\,\sinh 2y}
&\text{if \;}xy\neq0\,,\\
\hspace{12mm}0 &\text{if \;}xy=0\,.
\end{cases}\end{equation}
Moreover, we define the kernel $\mathcal{K}_{\alpha,\alpha}$ by
\begin{align}\label{Kalphaalpha}
\mathcal{K}_{\alpha,\alpha}(x,y,z)
&=2^{\ssf4\alpha+2}M_{\alpha,\alpha}\,
e^{\hspace{.5mm}x+y-z}\,\nonumber\\
&\times\frac{\bigl[\ssf\sinh(x\!+\!y\!+\!z)\sinh(\!-x\!+\!y\!+\!z)
\sinh(x\!-\!y\!+\!z)\sinh(x\!+\!y\!-\!z)\ssf\bigr]^{\alpha-1/2}}
{\bigl|\,\sinh2x\,\sinh2y\,\sinh2z\,\bigr|^{2\alpha}}\nonumber\\
&\times\,
\frac{\sinh(x\!+\!y\!+\!z)\sinh(\!-x\!+\!y\!+\!z)\sinh(x\!-\!y\!+\!z)}
{\sinh2x\,\sinh2y\,\sinh2z}
\end{align}
if $||x|-|y||<|z|<|x|+|y|$, and
$\mathcal{K}_{\alpha,\alpha}(x,y,z)=0$ otherwise. The symmetry
properties of $\mathcal{K}_{\alpha,\beta}$ (see Remark
\ref{symmetry}) remain true for $\mathcal{K}_{\alpha,\alpha}$.

\begin{theorem}\label{productformulaalpha=beta}
In the case $\alpha=\beta>-\frac{1}{2}$, the product formula reads
\begin{equation}\label{prodalpha}
{\rm G}_\lambda^{(\alpha,\alpha)}(x)\,{\rm
G}_\lambda^{(\alpha,\alpha)}(y) =\int_{-\infty}^{+\infty}\! {\rm
G}_\lambda^{(\alpha,\alpha)}(z)\, d\mu_{\;x,y}^{(\alpha,\alpha)}(z),
\end{equation}
for $x,y\in\R$ and $\lambda\in \mathbb C.$ Here
\begin{equation}\label{measure2}
d\mu_{\;x,y}^{(\alpha,\alpha)}(z)=\begin{cases}
\,\mathcal{K}_{\alpha,\alpha}(x,y,z)\,A_{\alpha,\alpha}(\vert
z\vert)\,dz
&\text{if \;}xy\neq0\,,\\
\qquad d\delta_x(z)
&\text{if \;}y=0\,,\\
\qquad d\delta_y(z) &\text{if \;}x=0\,.
\end{cases}\end{equation}
\end{theorem}

\begin{proof}
The even product formula
\begin{equation*}
{\rm G}_{\lambda,\text{\rm e}}^{(\alpha,\alpha)}(x)\,{\rm
G}_{\lambda,\text{\rm e}}^{(\alpha,\alpha)}(y)
=\frac12\int_{I_{x,y}}{\rm G}_\lambda^{(\alpha,\alpha)}(z)\,
W_{\alpha,\alpha}(|x|,|y|,|z|)\,A_{\alpha,\alpha}(|z|)\,dz
\end{equation*}
and the mixed product formulae
\begin{align*}
{\rm G}_{\lambda,\text{\rm e}}^{(\alpha,\alpha)}(x)\,{\rm
G}_{\lambda,\text{\rm o}}^{(\alpha,\alpha)}(y)
&=\frac12\int_{I_{x,y}}{\rm
G}_\lambda^{(\alpha,\alpha)}(z)\,\sigma_{z,y,x}\,
W_{\alpha,\alpha}(|x|,|y|,|z|)\,A_{\alpha,\alpha}(|z|)\,dz\,,\\
{\rm G}_{\lambda,\text{\rm o}}^{(\alpha,\alpha)}(x)\,{\rm
G}_{\lambda,\text{\rm e}}^{(\alpha,\alpha)}(y)
&=\frac12\int_{I_{x,y}}{\rm
G}_\lambda^{(\alpha,\alpha)}(z)\,\sigma_{x,z,y}\,
W_{\alpha,\alpha}(|x|,|y|,|z|)\,A_{\alpha,\alpha}(|z|)\,dz
\end{align*}
are obtained as in the case $\alpha\!>\!\beta$. Here
$W_{\alpha,\alpha}$ is given by \eqref{Walpha}, $I_{x,y}$ by
\eqref{I} and $\sigma_{x,z,y}$ by \eqref{sigma2-fun}. As far as they
are concerned, odd products are splitted up as in \eqref{I1+I2}\,:
\begin{equation*}
{\rm G}_{\lambda,\text{\rm o}}^{(\alpha,\alpha)}(x)\,{\rm
G}_{\lambda,\text{\rm o}}^{(\alpha,\alpha)}(y)
=\mathcal{I}_{\lambda,1}^{(\alpha,\alpha)}(x,y)
+\mathcal{I}_{\lambda,2}^{(\alpha,\alpha)}(x,y).
\end{equation*}
The first expression $\mathcal{I}_{\lambda,1}^{(\alpha,\alpha)}$ is
handled as $\mathcal{I}_{\lambda,1}^{(\alpha,\beta)}$ in the case
$\alpha\!>\!\beta$. We perform this time the change of variables
$]0,\pi[\,\ni\!\psi\longmapsto z\!\in\,]0,+\infty[$ defined by
$\cosh z=|\gamma(x,y,1,\psi)|$ and we obtain this way
\begin{align*}
\mathcal{I}_{\lambda,1}^{(\alpha,\alpha)}(x,y)
&=-\,M_{\alpha,\alpha}\int_{0}^{\pi} {\rm G}_{\lambda,\text{\rm
e}}^{(\alpha,\alpha)}(\arg\cosh|\gamma(x,y,1,\psi)|)\,
\cos\psi\,(\sin\psi)^{2\alpha}\,d\psi\\
&=-\int_{|x-y|}^{x+y} {\rm G}_{\lambda,\text{\rm
e}}^{(\alpha,\alpha)}(z)\,\sigma_{x,y,z}\,
W_{\alpha,\alpha}(x,y,z)\,A_{\alpha,\alpha}(z)\,dz\,,
\end{align*}
hence
$$
\mathcal{I}_{\lambda,1}^{(\alpha,\alpha)}(x,y)
=-\,\frac12\int_{I_{x,y}} {\rm
G}_\lambda^{(\alpha,\alpha)}(z)\,\sigma_{x,y,z}\,
W_{\alpha,\alpha}(|x|,|y|,|z|)\,A_{\alpha,\alpha}(|z|)\,dz\,,
$$
first for $x,y\!>\!0$ and next for $x,y\!\in\!\R^*$. According to
the product formula for $\varphi_{\lambda}^{(\alpha+1,\alpha+1)}$,
the second expression $\mathcal{I}_{\lambda,2}^{(\alpha,\alpha)}$
becomes
\begin{align*}
\mathcal{I}_{\lambda,2}^{(\alpha,\alpha)}(x,y)
&=\frac{2\alpha\!+\!1}{2(\alpha\!+\!1)}
\int_{|x-y|}^{x+y}{\rm G}_{\lambda,\text{\rm o}}^{(\alpha,\alpha)}(z)\,\\
&\times\,\frac{\sinh2x\,\sinh2y}{\sinh2z}\,
W_{\alpha+1,\alpha+1}(x,y,z)\,A_{\alpha+1,\alpha+1}(z)\,dz
\end{align*}
for all $x,y\!>\!0$\,. By using
\begin{align*}
&W_{\alpha+1,\alpha+1}(x,y,z)=16\,\frac{\alpha+1}{\alpha+1/2}\\
&\times\frac
{\sinh(x\!+\!y\!+\!z)\sinh(\!-x\!+\!y\!+\!z)\sinh(x\!-\!y\!+\!z)\sinh(x\!+\!y\!-\!z)}
{\sinh^2\hspace{-.5mm}2x\sinh^2\hspace{-.5mm}2y\sinh^2\hspace{-.5mm}2z}
\,W_{\alpha,\alpha}(x,y,z)
\end{align*}
and
$$
A_{\alpha+1,\alpha+1}(z)
=\frac{\sinh^2\hspace{-.5mm}2z}4\,A_{\alpha,\alpha}(z)\,,
$$
we obtain
\begin{align*}
\mathcal{I}_{\lambda,2}^{(\alpha,\alpha)}(x,y)
&=2\int_{I_{x,y}}{\rm G}_\lambda^{(\alpha,\alpha)}(z)\\
&\times\frac
{\sinh(x\!+\!y\!+\!z)\sinh(\!-x\!+\!y\!+\!z)\sinh(x\!-\!y\!+\!z)\sinh(x\!+\!y\!-\!z)}
{\sinh2x\sinh2y\sinh2z}\\
&\times
W_{\alpha,\alpha}(|x|,|y|,|z|)\,A_{\alpha,\alpha}(|z|)\,dz\,,
\end{align*}
first for $x,y\!>\!0$ and next for $x,y\!\in\!\R^*$. We conclude the
proof of Theorem \ref{productformulaalpha=beta} by summing all
partial product formulas and by using the remarkable identity
\begin{align*}
\varrho(x,y,z)
&=1-\sigma_{x,y,z}+\sigma_{z,y,x}+\sigma_{x,z,y}\\
&=4\,\frac
{\sinh(x\!+\!y\!+\!z)\sinh(\!-x\!+\!y\!+\!z)\sinh(x\!-\!y\!+\!z)\cosh(x\!+\!y\!-\!z)}
{\sinh2x\sinh2y\sinh2z}\,.
\end{align*}
\end{proof}

Consider next the rational limit of the product formula
\eqref{prodalpha}. It is well known that the hypergeometric function
${}_2{\rm F}_{\!1}\bigl(a,b\ssf;c\ssf;z\bigr)$ tends to the
confluent hypergeometric limit function ${}_0{\rm
F}_{\!1}\bigl(c\ssf;Z\bigr)$ as $a,b\!\to\!\infty$ and $z\!\to\!0$
in such a way that $a\ssf b\ssf z\!\to\!Z$. Consequently, as
$\varepsilon\!\to\!0$,
$$
\varphi_{\lambda/\varepsilon}^{(\alpha,\alpha)}(\varepsilon x)
={}_2{\rm
F}_{\!1}\Bigl(\alpha\!+\!\frac12\!+\!i\frac\lambda{2\ssf\varepsilon},
\alpha\!+\!\frac12\!-\!i\frac\lambda{2\ssf\varepsilon}\ssf;
\alpha\!+\!1\ssf;-(\sinh\varepsilon x)^2\Bigr)
$$
tends to the normalized Bessel function
$$
{\rm j}_{\alpha}(\lambda x) ={}_0{\rm
F}_{\!1}\Bigl(\alpha\!+\!1\ssf; -\bigl(\frac{\lambda\ssf
x}2\bigr)^2\Bigr) =\Gamma(\alpha\!+\!1)\sum_{ m=0}^{+\infty}
\frac{(-1)^m}{m\ssf!\,\Gamma(\alpha\!+\!1\!+\!m)}\,
\bigl(\frac{\lambda\ssf x}2\bigr)^{2m},
$$
hence
$$
{\rm G}_{\lambda/\varepsilon}^{(\alpha,\alpha)}(\varepsilon x)
=\varphi_{\lambda/\varepsilon}^{(\alpha,\alpha)}(\varepsilon x)
+\frac{2\alpha\!+\!1\!+\!i
{\lambda}/{\varepsilon}}{4\,(\alpha\!+\!1)} \sinh(2\varepsilon x)\,
\varphi_{\lambda/\varepsilon}^{(\alpha,\alpha)}(\varepsilon x)
$$
tends to
\begin{equation*}
{\rm E}_{\,\alpha}(i\lambda,x)={\rm j}_\alpha(\lambda x)
+\frac{i\lambda\ssf x}{2\,(\alpha\!+\!1)}\,{\rm
j}_{\alpha+1}(\lambda x).
\end{equation*}
The latter expression is the so--called Dunkl kernel in dimension~1,
whose product formula was obtained in \cite{Roes}\,:
\begin{equation}\label{prodDunkl}
{\rm E}_{\,\alpha}(i\lambda, x)\,{\rm E}_{\,\alpha}(i\lambda, y)
=\int_{\R}{\rm E}_{\,\alpha}(i\lambda,
z)\,k_\alpha(x,y,z)\,|z|^{2\alpha+1}\,dz\,,
\end{equation}
where
\begin{align*}
k_{\alpha}(x,y,z)&=2^{-2\alpha}\,
\frac{\Gamma(\alpha\!+\!1)}{\sqrt{\pi}\,\Gamma(\alpha\!+\!1/2)}\,
\bigl[\,1-\varsigma_{\,x,y,z}+\varsigma_{\,z,y,x}+\varsigma_{\,x,z,y}\,\bigr]\\
&\times\frac{[\,(x\!+\!y\!+\!z)\,(\!-x\!+\!y\!+\!z)\,
(x\!-\!y\!+\!z)\,(x\!+\!y\!-\!z)\,]^{\,\alpha-\frac12}}
{|\,x\,y\,z\,|^{\,2\alpha}}\nonumber\\
\end{align*}
with
$$
\varsigma_{x,y,z}=\,\begin{cases} \,\frac{x^2+\,y^2-\,z^2}{2\,x\,y}
&\text{if \;}x\,y\neq0\,,\\
\qquad0
&\text{if \;}x\,y=0\,,\\
\end{cases}
$$
hence
\begin{align}\label{kalpha}
k_{\alpha}(x,y,z)&=2^{-2\alpha-1}\,
\frac{\Gamma(\alpha\!+\!1)}{\sqrt{\pi}\,\Gamma(\alpha\!+\!1/2)}\,
\frac{[(x\!+\!y\!+\!z)(\!-x\!+\!y\!+\!z)
(x\!-\!y\!+\!z)(x\!+\!y\!-\!z)]^{\alpha-\frac12}}
{|x\hspace{.1mm}yz|^{2\alpha}}\nonumber\\
&\times\,\frac{(x\!+\!y\!+\!z)(\!-x\!+\!y\!+\!z)(x\!-\!y\!+\!z)}
{x\hspace{.1mm}yz}\,.
\end{align}
Here is an immediate consequence of \eqref{Kalphaalpha} and
\eqref{kalpha}.

\begin{lemma}\label{limitkernel}
For every \,$\alpha\!>\!-\frac12$ and $x,y,z\!\in\!\R^*$, we have
$$
\lim_{\varepsilon\to0}\,\varepsilon^{\,2\alpha+2}\,
\mathcal{K}_{\alpha,\alpha}(\varepsilon x,\varepsilon y,\varepsilon
z) =k_{\alpha}(x,y,z).
$$
\end{lemma}

We deduce the following result, which was announced in the abstract
and in the introduction.

\begin{corollary}
The product formula \eqref{prodDunkl} is the rational limit of the
product formula \eqref{prodalpha}. More precisely, \eqref{prodDunkl}
is obtained by replacing $\lambda$ by $\lambda/\varepsilon$ and
$(x,y)$ by $(\varepsilon x,\varepsilon y)$ in \eqref{prodalpha}, and
by letting $\varepsilon\to0$.
\end{corollary}

\begin{theorem}\label{mu}
Let  $x,y \in\R.$
\begin{itemize}
\item[(i)]
For \,$\alpha\!\ge\!\beta\!\ge\!-\frac12$ with
\,$\alpha\!>\!-\frac12$, we have
\,$\supp\mu_{\,x,y}^{(\alpha,\beta)}\!\subset\!I_{x,y}$.
\item[(ii)]
For \,$\alpha\!\ge\!\beta\!\ge\!-\frac12$ with
\,$\alpha\!>\!-\frac12$, we have
\,$\mu_{\,x,y}^{(\alpha,\beta)}(\R)=1$.
\item[(iii)]
For \,$\alpha\!>\!\beta\!>\!-\frac12$, we have
\,$\|\ssf\mu_{\,x,y}^{(\alpha,\beta)}\|
\le4+\displaystyle{\frac{\Gamma(\alpha+1)}{\Gamma(\alpha+\frac12)}
\frac{\Gamma(\beta+\frac12)}{\Gamma(\beta+1)}}$.
\item[(iv)]
For \,$\alpha\!=\!\beta\!>\!-\frac12$, we have
\,$\|\ssf\mu_{x,y}^{(\alpha,\alpha)}\|\leq\frac{5}{2}.$
\end{itemize}
\end{theorem}

\begin{proof}
{\rm(i)} is obvious.\\
{\rm(ii)} This  claim follows from Theorems 3.2 and 3.5
and the fact that ${\rm G}_{i\rho}^{(\alpha,\beta)}\equiv1.$\\
{\rm(iii)}  From the proof of Theorem 3.2, we may rewrite the
product formula for ${\rm G}_\lambda^{(\alpha,\beta)}$ as follows:
\begin{equation*}
{\rm G}_\lambda^{(\alpha,\beta)}(x)\,{\rm
G}_\lambda^{(\alpha,\beta)}(y) =\int_{I_{x,y}}{\rm
G}_\lambda^{(\alpha,\beta)}(z)\,
\widetilde{\mathcal{K}}_{\alpha,\beta}(x,y,z)\,A_{\alpha,\beta}(|z|)\,dz\,
+\,\mathcal{I}_{\lambda,2}^{(\alpha,\beta)}(x,y),
\end{equation*}
where $\mathcal{I}_{\lambda,2}^{(\alpha,\beta)}$ is given by
\eqref{I-2} and
\begin{align*}
\widetilde{\mathcal{K}}_{\alpha,\beta}(x,y,z)
:\!&=M_{\alpha,\beta}\;\bigl|\,\sinh x\,\sinh y\,\sinh z\,\bigr|^{-2\alpha}\\
&\times\int_0^\pi
(1-\sigma_{x,y,z}^\chi+\sigma_{x,z,y}^\chi+\sigma_{z,y,x}^\chi)\,
g(x,y,z,\chi)_+^{\alpha-\beta-1}\, (\sin\chi)^{2\beta}\,d\chi.
\end{align*}
By \cite[Proposition 2.7]{BO}, we have
$$
\int_{I_{x,y}}|\,\widetilde{\mathcal{K}}_{\alpha,\beta}(x,y,z)|\,
A_{\alpha, \beta}(|z|)\,dz\leq4.
$$
On the other hand, using the product formula \eqref{formula prod1}
for the Jacobi functions, we may rewrite
$\mathcal{I}_{\lambda,2}^{(\alpha,\beta)}$ as follows :
\begin{align*}
&\mathcal{I}_{\lambda,2}^{(\alpha,\beta)}(x,y)
=\frac{\rho(\rho\!+\!i\lambda)}{8(\alpha\!+\!1)^2}\,\sinh2x\,\sinh2y\;
\varphi_\lambda^{(\alpha+1,\beta+1)}(x)\,\varphi_\lambda^{(\alpha+1,\beta+1)}(y)\\
&=\frac{\rho(\rho\!+\!i\lambda)}{8(\alpha\!+\!1)^2}
\sinh2x\sinh2y\!\int_{0}^{1}\!\!\int_{0}^{\pi}\!\!
\varphi_\lambda^{(\alpha+1,\beta+1)}(
\arg\cosh|\gamma(x,y,r,\psi)|)\,
dm_{\alpha+1,\beta+1}(r,\psi)\\
&=\frac{\rho}{4(\alpha\!+\!1)}
\sinh2x\sinh2y\!\int_{\,0}^{\ssf1}\!\!\int_{\,0}^{\ssf\pi}\!\!
\frac{{\rm G}_{\lambda,\text{\rm
o}}^{(\alpha,\beta)}(\arg\cosh|\gamma(x,y,r,\psi)|)}
{|\gamma(x,y,r,\psi)|\sqrt{|\gamma(x,y,r,\psi)|^2-1}}\,
dm_{\alpha+1,\beta+1}(r,\psi),
\end{align*}
where $\gamma(x,y,r,\psi)=\cosh x\cosh y+\sinh x\sinh
y\,re^{i\psi}.$ In order to conclude, it remains  for us to prove
the following inequality
$$
\frac{\rho}{4(\alpha\!+\!1)}\sinh2x\sinh2y\!\int_{0}^{1}\!\!\int_{0}^{\pi}\!\!\frac{dm_{\alpha+1,\beta+1}(r,\psi)}{|\gamma(x,y,r,\psi)|
\sqrt{|\gamma(x,y,r,\psi)|^2\!-\!1}}\leq
\frac{\Gamma(\alpha\!+\!1)}{\Gamma(\alpha\!+\!\frac12)}
\frac{\Gamma(\beta\!+\!\frac12)}{\Gamma(\beta\!+\!1)}.
$$
By expressing $|\gamma(x,y,r,\psi)|$ and $dm_{\alpha+1,\beta+1},$
the left hand side becomes
\begin{align*}
&\frac{\rho}{4(\alpha\!+\!1)}\sinh2x\,\sinh2y\!\int_{0}^{1}\!\!\int_{0}^{\pi}\!\!
\frac{dm_{\alpha+1,\beta+1}(r,\psi)}{|\gamma(x,y,r,\psi)|
\sqrt{|\gamma(x,y,r,\psi)|^2-1}}\\
&=\frac\rho{4(\alpha\!+\!1)}\frac{2\,\Gamma(\alpha\!+\!2)}
{\sqrt{\pi}\,\Gamma(\alpha\!-\!\beta)\,\Gamma(\beta\!+\!\frac32)}\,
\sinh2x\,\sinh2y\!\int_{\,0}^{\ssf1}\!\!\int_{\,0}^{\ssf\pi}\!
(1\!-\!r^2)^{\alpha-\beta-1}(r\sin\psi)^{2\beta+2}\\
&\times\,\frac1{\sqrt{(\cosh x\cosh y\!+\!r\cos\psi\sinh x\sinh
y)^2\!
+\!(r\sin\psi\sinh x\sinh y)^2}}\\
&\times\frac1{\sqrt{(\cosh x\cosh y\!+\!r\cos\psi\sinh x\sinh y)^2\!
+\!(r\sin\psi\sinh x\sinh y)^2\!-\!1}}\,r\,dr\,d\psi\\
&=\frac{\rho\,\Gamma(\alpha\!+\!1)}
{\sqrt{\pi}\,\Gamma(\alpha\!-\!\beta)\,\Gamma(\beta\!+\!\frac{3}{2})}
\int_{0}^{1}\!\!\int_{0}^{\pi}\!\!(1\!-\!r^2)^{\alpha-\beta-1}(r\sin\psi)^{2\beta+2}
\frac{dr\,d\psi}{\sqrt{U\!+\!\cos\psi}\sqrt{V\!+\!\cos\psi}},
\end{align*}
where
$$
U=\frac{\cosh^2\!x\cosh^2\!y+r^2\sinh^2\!x\sinh^2\!y} {2r\cosh
x\cosh y\sinh x\sinh y},
$$
and
$$
V=\frac{\cosh^2\!x\cosh^2\!y+r^2\sinh^2\!x\sinh^2\!y-1} {2r\cosh
x\cosh y\sinh x\sinh y}.
$$
Since
$$
U-1>V-1 =\frac{(\cosh x\cosh y-r\sinh x\sinh y)^2-1} {2r\cosh x\cosh
y \sinh x\sinh y}\geq0,
$$
we can estimate
\begin{align*}
&\frac{\rho}{\alpha\!+\!1}\sinh x\cosh x\sinh y\cosh y
\int_{0}^{1}\!\!\int_{0}^{\pi}\!\!
\frac{dm_{\alpha+1,\beta+1}(r,\psi)}
{|\gamma(x,y,r,\psi)|\sqrt{|\gamma(x,y,r,\psi)|^2-1}}\\
&\leq\frac{\rho\,\Gamma(\alpha\!+\!1)}
{\sqrt{\pi}\,\Gamma(\alpha\!-\!\beta)\,\Gamma(\beta\!+\!\frac{3}{2})}
\int_{0}^{1}\!\!\int_{0}^{\pi}\!
(1-r^2)^{\alpha-\beta-1}(r\sin\psi)^{2\beta+2}(1+\cos\psi)^{-1}dr d\psi\\
&=\frac{\rho}{2}\,
\frac{\Gamma(\alpha\!+\!1)}{\Gamma(\alpha\!+\!\frac{3}{2})}\,
\frac{\Gamma(\beta\!+\!\frac{1}{2})}{\Gamma(\beta\!+\!1)}
\leq\frac{\Gamma(\alpha\!+\!1)}{\Gamma(\alpha\!+\!\frac{1}{2})}\,
\frac{\Gamma(\beta\!+\!\frac{1}{2})}{\Gamma(\beta\!+\!1)}\,,
\end{align*}
using classical formulas for the Beta and Gamma functions.\\
{\rm(iv)} is proved in a similar way, using the product formula
\eqref{formula prod11} for $\varphi_\lambda^{(\alpha, \alpha)}$
instead of \eqref{formula prod1}.
\end{proof}

\begin{remark}
The measure \,$\mu_{\,x,y}^{(\alpha,\beta)}$ is not positive, for
any \,$\alpha\!\ge\!\beta\!>\!-\frac12$ and \,$x,y\!>\!0$. More
precisely, let us show that \,$\mathcal{K}_{\alpha,\beta}(x,y,z)<0$
if \,$-x\!-\!y<z<-|x\!-\!y|$, while
\,$\mathcal{K}_{\alpha,\beta}(x,y,z)>0$ if \,$|x\!-\!y|<z<x\!+\!y$.
In the limit case \,$\alpha\!=\!\beta$, our claim follows
immediately from the expression \eqref{Kalphaalpha}. Thus we may
restrict to the case \,$\alpha\!>\!\beta$. Assume first that
\,$-x\!-\!y<z<-|x\!-\!y|$ and let us split up
\begin{equation*}
\mathcal{K}_{\alpha,\beta}(x,y,z) =M_{\alpha,\beta}\, (\,\sinh
x\,\sinh y\,\sinh(-z))^{-2\alpha}
\Big[\mathcal{K}_{\alpha,\beta}^{(1)}(x,y,z)
+\mathcal{K}_{\alpha,\beta}^{(2)}(x,y,z)\Big],
\end{equation*}
where
$$
\mathcal{K}_{\alpha,\beta}^{(1)}(x,y,z)
=\!\int_{0}^{\pi}\!g(x,y,-z,\chi)_{+}^{\alpha-\beta-1}
\bigl(1\!-\sigma_{x,y,z}^{\chi}\!+\sigma_{x,z,y}^{\chi}\!+\sigma_{z,y,x}^{\chi}\bigr)
(\sin\chi)^{2\beta}\,d\chi
$$
and
$$
\mathcal{K}_{\alpha,\beta}^{(2)}(x,y,z)
=\frac{\rho}{\beta\!+\!1/2}\,\coth x\coth y\coth z
\int_{0}^{\pi}\!g(x,y,-z,\chi)_{+}^{\alpha-\beta-1}(\sin\chi)^{2\beta+2}\,d\chi\,.
$$
On one hand, \,$\coth x\,\coth y\,\coth z<-1$ and
\begin{equation}\label{positive}
\int_{0}^{\pi}g(x,y,-z,\chi)_{+}^{\alpha-\beta-1}\,(\sin\chi)^{2\beta+2}\,d\chi>0\,,
\end{equation}
as the change of variables \eqref{ChangeVariables} holds for
\,$\chi$ in an interval starting at \,$0$, where
$$
g(x,y,-z,\chi)=\sinh^2\!x\,\sinh^2\!y\,(1\!-\hspace{-.25mm}r^2)>0\,.
$$
Hence \,$\mathcal{K}_{\alpha,\beta}^{(2)}(x,y,z)<0$. On the other
hand, as
\begin{align*}\textstyle
\varrho^\chi(x,y,z) &\textstyle
=1-\sigma_{x,y,z}^{\chi}+\sigma_{x,z,y}^{\chi}+\sigma_{z,y,x}^{\chi}\\
&\textstyle =\frac1{\sinh x\sinh y\sinh z}\,
\bigl[\,\sinh x\sinh y\sinh z+\sinh x\cosh y\cosh z\\
&\textstyle\hspace{29mm}
+\cosh x\sinh y\cosh z-\cosh x\cosh y\sinh z\\
&\textstyle\hspace{29mm}
+\frac{\cos\chi}2(-\sinh2x\hspace{-.25mm}-\hspace{-.25mm}\sinh2y\hspace{-.25mm}+\hspace{-.25mm}\sinh2z)\,\bigr]
\end{align*}
is a decreasing function of \,$\chi$, we have
\begin{align*}\textstyle
\varrho^\chi(x,y,z) &\textstyle
\le\varrho^0(x,y,z)\\
&\textstyle =\frac4{\sinh x\sinh y\sinh z}\,
\sinh\frac{x+y+z}2\,\sinh\frac{-x+y+z}2\,\sinh\frac{x-y+z}2\,\cosh\frac{x+y-z}2\\
&\textstyle <0\,.
\end{align*}
Hence \,$\mathcal{K}_{\alpha,\beta}^{(2)}(x,y,z)<0$. When
\,$|x\!-\!y|<z<x+y$, the positivity of
\,$\mathcal{K}_{\alpha,\beta}(x,y,z)$ is proved along the same
lines. If \,$\sinh2z\le\sinh2x+\sinh2y$, we have now
$$
\varrho^\chi(x,y,z)\ge\varrho^0(x,y,z)>0
$$
while, if \,$\sinh2z\ge\sinh2x+\sinh2y$,
\begin{align*}\textstyle
\varrho^\chi(x,y,z) &\textstyle
\ge\varrho^\pi(x,y,z)\\
&\textstyle =\frac4{\sinh x\sinh y\sinh z}\,
\cosh\frac{x+y+z}2\,\cosh\frac{-x+y+z}2\,\cosh\frac{x-y+z}2\,\sinh\frac{x+y-z}2\\
&\textstyle
>0\,.
\end{align*}
\end{remark}

\section{Generalized translations and convolution product}

Let us denote by $\mathcal{C}_{c}(\R)$ the space of continuous
functions on $\mathbb{R}$ with compact support.

  Let
$\alpha\ge\beta\ge-\frac12$ with $\alpha>-\frac12$. The
Opdam--Cherednik transform is the Fourier transform in the
trigonometric Dunkl setting. It is defined for
$f\!\in\ssb\mathcal{C}_{c}(\R)$ by
\begin{equation}
\label{fourier} \mathcal{F}(f)(\lambda)=\int_{\R} f(x)\,{\rm
G}_\lambda^{(\alpha,\beta)}(-x)\,A_{\alpha,\beta}(|x|)\,dx
\qquad\forall\;\lambda\!\in\!\C
\end{equation}
and the inverse transform writes
\begin{equation*}
\mathcal{J}g(x)=\int_{\R} g(\lambda)\,{\rm
G}_\lambda^{(\alpha,\beta)}(x)\,
\Bigl(1-\frac{\rho}{i\lambda}\Bigr)\,
\frac{d\lambda}{8\pi\,|\c_{\alpha,\beta}(\lambda)|^2}\,.
\end{equation*}
Here $A_{\alpha,\beta}$ and $\c_{\alpha,\beta}$ are given by
\eqref{A-fun} and \eqref{cfunction}. See \cite{Op} for more details.

The Fourier transform $\mathcal{F}$ can be expressed in terms of the
Jacobi transform
\begin{equation}\label{jacobitrans}
\mathcal{F}_{\alpha,\beta}(f)(\lambda)\,=\int_{\,0}^{+\infty}\!
f(x)\,\varphi_{\lambda}^{(\alpha,\beta)}(x)\,A_{\alpha,\beta}(x)\,dx\,.
\end{equation}
More precisely:
\begin{lemma}
For   $\lambda\in\C$ and $f\in C_c(\R)$,  we have
$$
\mathcal{F}(f)(\lambda) =2\,\mathcal{F}_{\alpha,\beta}(f_{\text{\rm
e}})(\lambda)
+2\,(\rho\!+\!i\lambda)\,\mathcal{F}_{\alpha,\beta}(J\hspace{-.4mm}f_{\text{\rm
o}})(\lambda)\,,
$$
where $f_{\text{\rm e}}$ (resp. $f_{\text{\rm o}}$) denotes the even
(resp. odd) part of $f$, and
$$
J\hspace{-.4mm}f_{\text{\rm
o}}(x):=\int_{-\infty}^{\,x}\!f_{\text{\rm o}}(t)\,dt\,.
$$
\end{lemma}
\begin{proof}
Write $f=f_{\text{\rm e}}+f_{\text{\rm o}}$\,. Firstly, if
\,$\lambda=- i\rho$\,, then
$$
\mathcal{F}(f)(\lambda) =\int_{\R}f(x)\,A_{\alpha,\beta}(|x|)\,dx
=2\,\mathcal{F}_{\alpha,\beta}(f_{\text{\rm e}})(i\rho)\,.
$$
Secondly, if \,$\lambda\!\ne\!-i\rho$\,, we have
$$
\mathcal{F}(f)(\lambda) =2\,\mathcal{F}_{\alpha,\beta}(f_{\text{\rm
e}})(\lambda)
+\frac{2}{\rho\!-\!i\lambda}\int_{\,0}^{+\infty}\!f_{\text{\rm
o}}(x)\, \frac{\partial}{\partial
x}\varphi_{\lambda}^{(\alpha,\beta)}(x)\, A_{\alpha,\beta}(x)\,dx\,.
$$
Recall the Jacobi operator
$$
\Delta_{\alpha,\beta}
=\frac1{A_{\alpha,\beta}(x)}\frac{\partial}{\partial x}
\Bigl[A_{\alpha,\beta}(x)\frac{\partial}{\partial x}\Bigr]
=\frac{\partial^2}{\partial x^2} +\Bigl[(2\alpha\!+\!1)\coth
x+(2\beta\!+\!1)\tanh x\Bigr] \frac{\partial}{\partial x}\,.
$$
By integration by parts, we obtain
\begin{align*}
&\int_{\,0}^{+\infty}\!f_{\text{\rm o}}(x)\,
\frac{\partial}{\partial x}\varphi_{\lambda}^{(\alpha,\beta)}(x)\,
A_{\alpha,\beta}(x)\,dx\\
&=-\int_{\,0}^{+\infty}\!\varphi_{\lambda}^{(\alpha,\beta)}(x)\,
\frac1{A_{\alpha,\beta}(x)}\,\frac{\partial}{\partial x}
\Bigl[A_{\alpha,\beta}(x)\,\frac{\partial}{\partial
x}\,J\hspace{-.4mm}f_{\text{\rm o}}(x)\Bigr]
A_{\alpha,\beta}(x)\,dx\\
&=-\,\mathcal{F}_{\alpha,\beta}\ssf (\ssf\Delta_{\alpha,\beta}\ssf
J\hspace{-.4mm}f_{\text{\rm o}})(\lambda)
=(\rho^2\!+\hspace{-.4mm}\lambda^2)\,
\mathcal{F}_{\alpha,\beta}\ssf(J\hspace{-.4mm}f_{\text{\rm
o}})(\lambda)\,.
\end{align*}
\end{proof}
The following Plancherel formula was proved by Opdam \cite[Theorem
9.13(3)]{Op}\,:
\begin{align*}
\int_{\R}|f(x)|^2A_{\alpha,\beta}(|x|)\,dx &=\int_{\,0}^{+\infty}\!
\bigl(\,|\mathcal{F}(f)(\lambda)|^2+|\mathcal{F}(\check{f})(\lambda)|^2\,\bigr)\,
\frac{d\lambda}{16\pi\,|\c_{\alpha,\beta}(\lambda)|^2}\\
&=\int_{\R}\mathcal{F}(f)(\lambda)\,
\overline{\mathcal{F}(\check{f})(-\lambda)}\,
(1-\frac{\rho}{i\lambda})\,
\frac{d\lambda}{8\pi\,|\c_{\alpha,\beta}(\lambda)|^2}\,,
\end{align*}
where $\check{f}(x):=f(-x)$. The following  result is obtained by
specializing \cite[Theorem 4.1]{Sch}.

\begin{theorem}
The Opdam--Cherednik transform \,$\mathcal{F}$ and its inverse
\,$\mathcal{J}$ are topological isomorphisms between the Schwartz
space \,$\mathcal{S}_{\alpha,\beta}(\R)=(\cosh
x)^{-\rho}\,\mathcal{S}(\R)$ and the Schwartz space
\,$\mathcal{S}(\R)$. Recall that
\,$\rho\!=\!\alpha\!+\!\beta\!+\!1$.
\end{theorem}

Let us denote by $\mathcal{C}_{b}(\R)$ the space of bounded
continuous functions on $\mathbb{R}$.
\begin{definition}
Let \,$x\!\in\!\R$ and let \,$f\in \mathcal{C}_{b}(\R)$. For
\,$\alpha\!\ge\!\beta\!\ge\!-\frac12$ with
\,$\alpha\!\ne\!-\frac12$, we define the generalized translation
operator \,$\tau_{\,x}^{(\alpha,\beta)}$ by
$$
\tau_{\,x}^{(\alpha,\beta)}f(y)\,=\int_{\R}f(z)\,d\mu_{\,x,y}^{(\alpha,\beta)}(z)\,,
$$
where \,$d\mu_{x,y}^{(\alpha, \beta)}$ is given by \eqref{measure1}
for \,$\alpha\!>\!\beta$, and by \eqref{measure2} for
\,$\alpha\!=\!\beta$.
\end{definition}
The following properties are clear. However for completeness we will
sketch their proof.

\begin{proposition}${}$
Let $\alpha\!\ge\!\beta\!\ge\!-\frac12$ with
\,$\alpha\!\ne\!-\frac12$,  \,$x,\; \!y\in\!\R$ and  \,$f\in
\mathcal{C}_{b}(\R)$. Then
\begin{itemize}
\item[(i)]
$\tau_{\,x}^{(\alpha,\beta)}f(y) =\tau_{\,y}^{(\alpha,\beta)}f(x)$.
\item[(ii)]
$\tau_{\,0}^{(\alpha,\beta)}f=f$.
\item[(iii)]
$\tau_{\,x}^{(\alpha,\beta)}\tau_{\,y}^{(\alpha,\beta)}
=\tau_{\,y}^{(\alpha,\beta)}\tau_{\,x}^{(\alpha,\beta)}$.
\item[(iv)]
$\tau_{\,x}^{(\alpha,\beta)}{\rm G}_\lambda^{(\alpha,\beta)}(y)
={\rm G}_\lambda^{(\alpha,\beta)}(x)\,{\rm
G}_\lambda^{(\alpha,\beta)}(y)$.

 If we suppose also that $f$ belongs
to $ \mathcal{C}_{c}(\R)$, then
\item[(v)]
$\mathcal{F}(\tau_{\,x}^{(\alpha,\beta)}f)(\lambda) ={\rm
G}_\lambda^{(\alpha,\beta)}(x)\,\mathcal{F}(f)(\lambda)$.
\item[(vi)]
${\rm T}^{(\alpha,\beta)}\,\tau_{\,x}^{(\alpha,\beta)}
=\tau_{\,x}^{(\alpha,\beta)}\,{\rm T}^{(\alpha,\beta)}$.
\end{itemize}
\end{proposition}

\begin{proof}
{\rm(i)} follows from the property
$\mathcal K_{\alpha,\beta}(x,y,z)=\mathcal K_{\alpha,\beta}(y,x,z)$.\\
{\rm(ii)} follows from the fact that
$\mathcal K_{\alpha,\beta}(0,y,z)=\delta_y(z).$\\
{\rm(iii)} follows from the fact that the function
$$
H(x_1,y_1,x_2,y_2):=\int_{\R} \mathcal K_{\alpha,\beta}(x_1,y_1,z)\,
\mathcal K_{\alpha,\beta}(x_2,y_2,z)\, A_{\alpha,\beta}(|z|)\,dz
$$
is symmetric in the four variables.\\
{\rm(iv)} follows from  the product formula for ${\rm G}_\lambda^{(\alpha,\beta)}.$\\
{\rm(v)} For $f\in\mathcal{C}_c(\R),$ we have
\begin{align*}
\mathcal{F}(\tau_{\,x}^{(\alpha,\beta)}f)(\lambda)
&=\int_{\R}\tau_{\,x}^{(\alpha,\beta)}f(y)\,
{\rm G}_\lambda^{(\alpha,\beta)}(-y)\,A_{\alpha,\beta}(|y|)\,dy\\
&=\int_{\R}\Bigl[\,\int_{\R}
f(z)\,\mathcal{K}_{\alpha,\beta}(x,y,z)\,A_{\alpha,\beta}(|z|)\,dz\Bigr]
{\rm G}_\lambda^{(\alpha,\beta)}(-y)\,A_{\alpha,\beta}(|y|)\,dy\\
&=\int_{\R}f(z)\Bigl[\,\int_{\R}{\rm
G}_\lambda^{(\alpha,\beta)}(-y)\,
\mathcal{K}_{\alpha,\beta}(x,y,z)\,A_{\alpha,\beta}(|y|)\,dy\Bigr]
A_{\alpha,\beta}(|z|)\,dz\,.
\end{align*}
Since
$\mathcal{K}_{\alpha,\beta}(x,y,z)=\mathcal{K}_{\alpha,\beta}(x,-z,-y)$,
it follows from the product formula that
\begin{align*}
\mathcal{F}(\tau_{\,x}^{(\alpha,\beta)}f)(\lambda) &={\rm
G}_\lambda^{(\alpha,\beta)}(x)
\int_{\R}f(z)\,{\rm G}_\lambda^{(\alpha,\beta)}(-z)\,A_{\alpha,\beta}(|z|)\,dz\\
&={\rm G}_\lambda^{(\alpha,\beta)}(x)\,\mathcal{F}(f)(\lambda).
\end{align*}
{\rm(vi)} This property follows from the injectivity of
$\mathcal{F}$ and the fact that $\tau_{\,x}^{(\alpha,\beta)}({\rm
T}^{(\alpha,\beta)}f)$ and ${\rm
T}^{(\alpha,\beta)}(\tau_{\,x}^{(\alpha,\beta)}f)$ have the same
Fourier transform, namely
$$
\lambda\longmapsto i\,\lambda\, {\rm
G}_\lambda^{(\alpha,\beta)}(x)\,\mathcal{F}(f)(\lambda)\,.
$$
\end{proof}

\begin{remark}
Generalized translations in the Dunkl setting were first introduced
by Trim\`eche, using transmutation operators. This approach is
resumed in \cite{M}, which deals with a generalization of Dunkl
analysis in dimension $1$.
\end{remark}

\begin{lemma}
Let $1\leq p\leq\infty$, $f\in L^p(\R,A_{\alpha,\beta}(|z|)dz)$ and
$x\in\R.$ Then
\begin{equation}\label{ingtrans}
\|\tau_{\,x}^{(\alpha,\beta)}f\|_p\leq C_{\alpha,\beta}\,\|f\|_p\,,
\end{equation}
where
\begin{equation}\label{cst}
C_{\alpha,\beta}=\begin{cases}
\,4+\frac{\Gamma(\alpha+1)}{\Gamma(\alpha+\frac{1}{2})}\,
\frac{\Gamma(\beta+\frac{1}{2})}{\Gamma(\beta+1)}
&\text{if \;}\alpha>\beta>-\frac12\,,\\
\qquad\frac52 &\text{if \;}\alpha=\beta>-\frac12\,.
\end{cases}\end{equation}
\end{lemma}
\begin{proof}
The inequality \eqref{ingtrans} follows from Theorem \ref{mu}. More
precisely, the cases $p=1$ and $p=\infty$ are elementary, while the
intermediate case $1<p<\infty$ is obtained by interpolation or by
using H\"older's inequality, as follows\,:
\begin{align*}
\|\tau_{\,x}^{(\alpha,\beta)}f\|_p^p
&\le\Bigl(\int_{\R}|\mathcal{K}_{\alpha,\beta}(x,y,z)|\,
A_{\alpha,\beta}(|z|)\,dz\Bigr)^{p-1}\\
&\times\int_{\R}\int_{\R}
|\mathcal{K}_{\alpha,\beta}(x,y,z)|\,|f(z)|^p\,
A_{\alpha,\beta}(|z|)\,A_{\alpha,\beta}(|y|)\,dz\,dy\\
&\leq C_{\alpha,\beta}^{\,p}\,\|f\|_p^p\,.
\end{align*}
\end{proof}

\begin{definition}
The convolution product of suitable functions \,$f$ and \,$g$ is
defined by
$$
(f\ast_{\alpha,\beta}g)(x)=\int_{\R}
\tau_{\,x}^{(\alpha,\beta)}f(-y)\,g(y)\,A_{\alpha,\beta}(|y|)\,dy\,.
$$
\end{definition}
\begin{remark}
It is clear that this convolution product is both commutative and
associative:
\begin{itemize}
\item[(i)] $f\ast_{\alpha,\beta}g=g\ast_{\alpha,\beta}f.$
\item[(ii)] $(f\ast_{\alpha,\beta}g)\ast_{\alpha,\beta}h=f\ast_{\alpha,\beta}(g\ast_{\alpha,\beta}h).$
\end{itemize}
\end{remark}

For every $a\!>\!0$, let us denote by $\mathcal{D}_a(\R)$ the space
of smooth functions on $\R$ which are supported in \,$[-a,a\,]$\,.

\begin{proposition}
Let \,$f\!\in\!\mathcal{D}_a(\R)$ and \,$g\!\in\!\mathcal{D}_b(\R)$.
Then \,$f\ast_{\alpha,\beta}g\in\mathcal{D}_{a+b}(\R)$ and
$$
\mathcal{F}(f\ast_{\alpha,\beta}g)(\lambda)
=\mathcal{F}(f)(\lambda)\,\mathcal{F}(g)(\lambda)\,.
$$
\end{proposition}
\begin{proof}
By definition we have
$$
\mathcal{F}(f\ast_{\alpha,\beta}g)(\lambda)=\int_{\R}\int_{\R}
\tau_{\,x}^{(\alpha,\beta)}f(-y)\,g(y)\,{\rm
G}_\lambda^{(\alpha,\beta)}(-x)\,
A_{\alpha,\beta}(|x|)\,A_{\alpha,\beta}(|y|)\,dx\,dy\,.
$$
Using the product formula for ${\rm G}_\lambda^{(\alpha,\beta)}$ and
Remark 3.1, we deduce that
\begin{align*}
\mathcal{F}(f\ast_{\alpha,\beta}g)(\lambda)
&=\int_{\R}f(z)\int_{\R}g(y)\int_{\R} {\rm
G}_\lambda^{(\alpha,\beta)}(-x)\,
\mathcal K_{\alpha,\beta}(-z,-y,-x)\\
&\times\,A_{\alpha,\beta}(|x|)\,
A_{\alpha,\beta}(|y|)\,A_{\alpha,\beta}(|z|)\,dx\,dy\,dz\\
&=\int_{\R}f(z)\,{\rm
G}_\lambda^{(\alpha,\beta)}(-z)\,A_{\alpha,\beta}(|z|)\,dz\,
\int_{\R}g(y)\,{\rm G}_\lambda^{(\alpha,\beta)}(-y)\,A_{\alpha,\beta}(|y|)\,dy\\
&=\mathcal{F}(f)(\lambda)\,\mathcal{F}(g)(\lambda)\,.
\end{align*}
\end{proof}

By standard arguments, the following statement follows from Lemma
4.5.

\begin{proposition}\label{Young}
Assume that \,$1\!\le\!p,q,r\!\le\!\infty$ satisfy
$\frac1p\!+\!\frac1q\!-\!1\!=\!\frac1r$\,. Then, for every
\,$f\!\in\!L^p(\R,A_{\alpha,\beta}(|x|)\ssf dx)$ and
\,$g\!\in\!L^q(\R,A_{\alpha,\beta}(|x|)\ssf dx)$\,, we have
\,$f\ast_{\alpha,\beta}g\in L^r(\R,A_{\alpha,\beta}(|x|)\ssf dx)$,
and
$$
\|\,f\ast_{\alpha,\beta}g\,\|_r \le
C_{\alpha,\beta}\,\|f\|_p\,\|g\|_q\,,
$$
where $C_{\alpha,\beta}$ is as in \eqref{cst}.
\end{proposition}

\section{The Kunze--Stein phenomenon}

This remarkable phenomenon was first observed by Kunze and Stein
\cite{KS} for the group $G=SL(2,\R)$ equipped with its Haar measure.
They proved that
$$
L^p(G)\ast L^2(G)\subset L^2(G) \quad\forall\;1\!\le\!p\!<\!2\,.
$$
By such an inclusion, we mean the existence of a constant
$C_p\!>\!0$ such that the following inequality holds\,:
$$
\|\,f\ast g\,\|_2\leq C_p\,\|f\|_p\,\|g\|_2
\quad\forall\;f\!\in\!L^p(G),\;\forall\;g\!\in\!L^2(G).
$$
This result was generalized by Cowling \cite{Cow} to all connected
noncompact semisimple Lie groups with finite center. We prove the
following analog in our setting (we understand that Trim\`eche has
recently extended this result to higher dimensions).

\begin{theorem}\label{KunzeStein}
Let  $1\leq p<2<q\leq\infty.$  Then
\begin{equation}\label{KunzeStein1}
L^p(\R, A_{\alpha,\beta}(|x|)\ssf dx) \ast_{\alpha,\beta} L^2(\R,
A_{\alpha,\beta}(|x|)\ssf dx) \subset L^2(\R,
A_{\alpha,\beta}(|x|)\ssf dx)
\end{equation}
and
\begin{equation}\label{KunzeStein2}
L^2(\R, A_{\alpha,\beta}(|x|)\ssf dx) \ast_{\alpha,\beta}L^2(\R,
A_{\alpha,\beta}(|x|)\ssf dx) \subset L^q(\R,
A_{\alpha,\beta}(|x|)\ssf dx)\,.
\end{equation}
\end{theorem}

\begin{proof}
{\rm(i)} Let $f, g\!\in\!\mathcal{C}_c(\R) $. Then, by the
Plancherel formula, we have
\begin{align*}
&\int_{\R}|(f\ast_{\alpha,\beta} g)|^2(x)\,A_{\alpha,\beta}(|x|)\,dx\\
&=\int_{\R^+}|\mathcal F({f\ast_{\alpha,\beta} g})(\lambda)|^2
\,\frac{d\lambda}{16\pi\,|\c(\lambda)|^2} +\int_{\R^+}|\mathcal
F(f\ast_{\alpha,\beta} g)\,\check{}\,(\lambda)|^2
\,\frac{d\lambda}{16\pi\,|\c(\lambda)|^2}\\
&\le\sup_{\lambda\in\R,\,w\in\{\pm 1\}}\!|\mathcal{F}(w\cdot
g)(\lambda)|^2\, \Bigl[\,\int_{\R^+}|\mathcal{F}(f)(\lambda)|^2
\frac{d\lambda}{16\pi\,|\c(\lambda)|^2}+ \int_{\R^+}|\mathcal
F({\check{f}})(\lambda)|^2
\frac{d\lambda}{16\pi\,|\c(\lambda)|^2}\,\Bigr]\\
&=\sup_{\lambda\in\R,\,w\in\{\pm 1\}}\!|\mathcal{F}(w\cdot
g)(\lambda)|^2\, \|f\|_2^2\,.
\end{align*}
Here we have used the fact that
\,$\mathcal{F}(f\ast_{\alpha,\beta}g)\check\,{}\! =\mathcal
F(\check{f})\,\mathcal F(\check{g})$\,. Next, if \,$1\!\le\!p\!<\!2$
\,and \,$2\!<\!q\!\le\!\infty$ \,are dual indices, we estimate
\begin{align*}
|\mathcal{F}(w\cdot g)(\lambda)| &\le\int_{\R}|g(wx)|\,|{\rm
G}_\lambda^{(\alpha,\beta)}(-x)|\,
A_{\alpha, \beta}(|x|)\ssf dx\\
&\le\|g\|_p\,\|{\rm G}_\lambda^{(\alpha,\beta)}\|_q
\end{align*}
using H\"older's inequality. We conclude by using Lemma
\ref{Schapira} below, which implies that \,$\|{\rm
G}_\lambda^{(\alpha,\beta)}\|_q$ \,is bounded uniformly in
\,$\lambda\!\in\!\R$\,.

\noindent{\rm(ii)} Let $f,g,k\!\in\!\mathcal{C}_c(\R)$. Using the
Cauchy-Schwartz inequality and \eqref{KunzeStein1}, we get
\begin{align*}
\Bigl|\int_{\R}(f\ast_{\alpha,\beta}g)(x)\,k(x)\,A_{\alpha,\beta}(|x|)\,dx\,\Bigr|
&\le C\,\|g\|_2\,\|\,f\ast\check{k}\,\|_2\\
&\le C_p\,\|f\|_2\,\|g\|_2\,\|k\|_p\,.
\end{align*}
Hence \,$\|\,f\ast_{\alpha,\beta}g\,\|_q\le
C_q\,\|f\|_2\,\|g\|_2$\,.
\end{proof}

\begin{lemma}\label{Schapira}
{\rm (i)} The function \,${\rm G}_{\,0}^{(\alpha,\beta)}$ is stricly
positive and is bounded above by
\begin{equation*}\begin{cases}
\,C\,(1\!+\!x)\,e^{-\rho\,x}
&\text{if \;}x\!\ge\!0\,,\\
\qquad C\,e^{\,\rho\,x}
&\text{if \;}x\!\le\!0\,.\\
\end{cases}\end{equation*}
{\rm (ii)} For every \,$\lambda\!\in\!\R$ and \,$x\!\in\!\R$\,, we
have
\begin{equation*}
\bigl|{\rm G}_\lambda^{(\alpha,\beta)}(x)\bigr|\le {\rm
G}_{\,0}^{(\alpha,\beta)}(x)\,.
\end{equation*}
\end{lemma}

\begin{proof}
These estimates are proved in full generality in \cite{Sch} (see
Lemma 3.1, Proposition 3.1.a and Theorem 3.2). For the reader's
convenience, we include a proof in dimension $1$.

\noindent (i) Firstly, by specializing \eqref{Jacobi} and
\eqref{Gfunction} for $\lambda\!=\!0$, we obtain
\begin{equation}\label{Jacobi0}\textstyle
\varphi_{\,0}^{(\alpha,\beta)}(x) ={}_2{\rm
F}_{\!1}\bigl(\frac\rho2,\frac{\alpha-\beta+1}2;
\alpha\!+\!1\ssf;\tanh^2\!x\bigr)\, (\cosh x)^{-\rho}
\end{equation}
and
\begin{equation}\label{G0}
{\rm G}_{\,0}^{(\alpha,\beta)}(x) =\varphi_{\,0}^{(\alpha,\beta)}(x)
+\frac{\frac\rho2}{\alpha\!+\!1}\sinh x\cosh x\,
\varphi_{\,0}^{(\alpha+1,\beta+1)}(x)\,.
\end{equation}
It is clear that \eqref{Jacobi0} is stricly positive, hence
\eqref{G0} when $x\!\ge\!0$. By looking more carefully at their
expansions, we observe that the expression
\begin{equation*}\textstyle
\Psi_1(x) ={}_2{\rm F}_{\!1}
\bigl(\frac\rho2,\frac{\alpha-\beta+1}2;
\alpha\!+\!1\ssf;\tanh^2\!x\Bigr)
=\displaystyle\sum\nolimits_{n=0}^{+\infty}
\frac{(\frac\rho2)_n\,(\frac{\alpha-\beta+1}2)_n}
{(\alpha\!+\!1)_n\,n\ssf!}\,(\tanh x)^{2n}
\end{equation*}
is stricly larger than the expression
\begin{align*}
\Psi_2(x) &=\frac{\frac\rho2}{\alpha\!+\!1}\;\textstyle {}_2{\rm
F}_{\!1}\bigl(\frac\rho2\!+\!1,\frac{\alpha-\beta+1}2;
\alpha+2\ssf;\tanh^2\!x\bigr)\\
&=\sum\nolimits_{n=0}^{+\infty}
\frac{(\frac\rho2)_{n+1}(\frac{\alpha-\beta+1}2)_n}
{(\alpha\!+\!1)_{n+1}\,n\ssf!}\,(\tanh x)^{2n}\,.
\end{align*}
Hence
\begin{align*}
{\rm G}_{\,0}^{(\alpha,\beta)}(x) =(\cosh x)^{-\rho}\,\bigl\{
\Psi_1(x)+\tanh x\,\Psi_2(x)\bigr\}
>(\cosh x)^{-\rho}\,\bigl\{\Psi_1(x)-\Psi_2(x)\bigr\}
\end{align*}
is strictly positive on \,$\R$\,. Secondly, by combining \eqref{G0}
with \eqref{Jacobi0Asymptotic}, we obtain
\begin{equation*}
{\rm G}_0^{(\alpha,\beta)}(x)
=\frac{2^{\ssf\rho+2}\,\Gamma(\alpha\!+\!1)}
{\Gamma(\frac\rho2)\,\Gamma(\frac{\alpha-\beta+1}2)}\,
x\,e^{-\rho\,x} +\mathcal{O}\bigl(e^{-\rho\,x}\bigr) \qquad\text{as
\;}x\!\to\!+\infty
\end{equation*}
and
\begin{equation*}
{\rm G}_0^{(\alpha,\beta)}(x)=\mathcal{O}\bigl(e^{\,\rho\,x}\bigr)
\qquad\text{as \;}x\!\to\!-\infty\,,
\end{equation*}
which yields the announced upper bounds.\\
(ii) Consider the quotient
\begin{equation*}\textstyle
{\rm Q}_{\ssf\lambda}^{(\alpha,\beta)}(x) =\frac{{\rm
G}_\lambda^{(\alpha,\beta)}(x)} {{\rm G}_0^{(\alpha,\beta)}(x)}\,.
\end{equation*}
By using the equation \eqref{eigen} for \,${\rm
G}_\lambda^{(\alpha,\beta)}$ and \,${\rm G}_0^{(\alpha,\beta)}$, we
obtain
\begin{equation*}\begin{aligned}
&\textstyle \frac\partial{\partial x}\,{\rm
Q}_{\ssf\lambda}^{(\alpha,\beta)}(x) =\frac{\frac\partial{\partial
x}\,{\rm G}_\lambda^{(\alpha,\beta)}(x)} {{\rm
G}_0^{(\alpha,\beta)}(x)} -{\rm
Q}_{\ssf\lambda}^{(\alpha,\beta)}(x)\, \frac{\frac\partial{\partial
x}\,{\rm G}_0^{(\alpha,\beta)}(x)} {{\rm G}_0^{(\alpha,\beta)}(x)}
\\&\textstyle
=\bigl\{(\alpha\!-\!\beta)\coth x+(2\ssf\beta\!+\!1)\coth2\ssf
x+\rho\ssf\bigr\} \,\frac{{\rm G}_0^{(\alpha,\beta)}(-x)}{{\rm
G}_0^{(\alpha,\beta)}(x)}\, \bigl\{{\rm
Q}_{\ssf\lambda}^{(\alpha,\beta)}(-x)\ssb -{\rm
Q}_{\ssf\lambda}^{(\alpha,\beta)}(x)\bigr\}
\\&\textstyle
+i\ssf\lambda\,{\rm Q}_{\ssf\lambda}^{(\alpha,\beta)}(x)\,.
\end{aligned}\end{equation*}
Hence
\begin{align*}\textstyle
\frac\partial{\partial x}\, \bigl|\ssf{\rm
Q}_{\ssf\lambda}^{(\alpha,\beta)}(x)\bigr|^2 &\textstyle=
2\;\Re\,\Bigl[\,\frac\partial{\partial x}\,{\rm
Q}_{\ssf\lambda}^{(\alpha,\beta)}(x) \,\overline{{\rm
Q}_{\ssf\lambda}^{(\alpha,\beta)}(x)}\,\Bigr]
\\&\textstyle
=-\,2\,\bigl\{(\alpha\!-\!\beta)\coth x
+(2\ssf\beta\!+\!1)\coth2\ssf x+\rho\ssf\bigr\} \,\frac{{\rm
G}_0^{(\alpha,\beta)}(-x)}{{\rm G}_0^{(\alpha,\beta)}(x)}
\\&\textstyle\times
\Bigl\{\,\bigl|\ssf{\rm
Q}_{\ssf\lambda}^{(\alpha,\beta)}(x)\bigr|^2\ssb
-\ssf\Re\,\Bigl[\ssf{\rm Q}_{\ssf\lambda}^{(\alpha,\beta)}(-x)\,
\overline{{\rm Q}_{\ssf\lambda}^{(\alpha,\beta)}(x)}\,\Bigr]\Bigr\}
\end{align*}
and
\begin{align*}\textstyle
\frac\partial{\partial x}\, \bigl|\ssf{\rm
Q}_{\ssf\lambda}^{(\alpha,\beta)}(-x)\bigr|^2 &\textstyle
=-\,2\,\bigl\{(\alpha\!-\!\beta)\coth x
+(2\ssf\beta\!+\!1)\coth2\ssf x-\rho\ssf\bigr\} \,\frac{{\rm
G}_0^{(\alpha,\beta)}(x)}{{\rm G}_0^{(\alpha,\beta)}(-x)}
\\&\textstyle\times
\Bigl\{\,\bigl|\ssf{\rm
Q}_{\ssf\lambda}^{(\alpha,\beta)}(-x)\bigr|^2\ssb
-\ssf\Re\,\Bigl[\ssf{\rm Q}_{\ssf\lambda}^{(\alpha,\beta)}(x)\,
\overline{{\rm
Q}_{\ssf\lambda}^{(\alpha,\beta)}(-x)}\,\Bigr]\ssf\Bigr\}\,.
\end{align*}
Thus, for every \ssf$x\!>\!0$\ssf, we have
\begin{equation}\begin{aligned}\label{Qplus}\textstyle
\frac\partial{\partial x}\, \bigl|\ssf{\rm
Q}_{\ssf\lambda}^{(\alpha,\beta)}(x)\bigr|^2 &\textstyle
\le-\,2\,\bigl\{(\alpha\!-\!\beta)\coth x
+(2\ssf\beta\!+\!1)\coth2\ssf x+\rho\ssf\bigr\} \, \frac{{\rm
G}_0^{(\alpha,\beta)}(-x)}{{\rm G}_0^{(\alpha,\beta)}(x)}
\\&\textstyle\times\,
\bigl|\ssf{\rm Q}_{\ssf\lambda}^{(\alpha,\beta)}(x)\bigr|\,
\bigl\{\bigl|\ssf{\rm Q}_{\ssf\lambda}^{(\alpha,\beta)}(x)\bigr|\ssb
-\ssb\bigl|\ssf{\rm Q}_{\ssf\lambda}^{(\alpha,\beta)}(-x)\bigr|\bigr\}\\
&\le\,0\vphantom{\frac00}
\end{aligned}\end{equation}
if \,$\bigl|\ssf{\rm Q}_{\ssf\lambda}^{(\alpha,\beta)}(x)\bigr|\ssb
\ge\ssb\bigl|\ssf{\rm Q}_{\ssf\lambda}^{(\alpha,\beta)}(-x)\bigr|$
\,and
\begin{equation}\begin{aligned}\label{Qminus}\textstyle
\frac\partial{\partial x}\, \bigl|\ssf{\rm
Q}_{\ssf\lambda}^{(\alpha,\beta)}(-x)\bigr|^2 &\textstyle
\le-\,2\,\bigl\{(\alpha\!-\!\beta)\coth x
+(2\ssf\beta\!+\!1)\coth2\ssf x-\rho\ssf\bigr\} \, \frac{{\rm
G}_0^{(\alpha,\beta)}(x)}{{\rm G}_0^{(\alpha,\beta)}(-x)}
\\&\textstyle\times\,
\bigl|\ssf{\rm Q}_{\ssf\lambda}^{(\alpha,\beta)}(-x)\bigr|\,
\bigl\{\bigl|\ssf{\rm
Q}_{\ssf\lambda}^{(\alpha,\beta)}(-x)\bigr|\ssb
-\ssb\bigl|\ssf{\rm Q}_{\ssf\lambda}^{(\alpha,\beta)}(x)\bigr|\bigr\}\\
&\le\,0\vphantom{\frac00}
\end{aligned}\end{equation}
if \,$|{\rm Q}_{\ssf\lambda}^{(\alpha,\beta)}(-x)|\ssb \ge\ssb|{\rm
Q}_{\ssf\lambda}^{(\alpha,\beta)}(x)|$\ssf. As real analytic
functions of \ssf$x$\ssf, \,$|\ssf{\rm
Q}_{\ssf\lambda}^{(\alpha,\beta)}(x)|^2$ and \,$|\ssf{\rm
Q}_{\ssf\lambda}^{(\alpha,\beta)}(-x)|^2$ coincide either everywhere
or on a discrete subset of \ssf$\R$ \ssf with no accumulation point.
In the first case, \ssf$|\ssf{\rm
Q}_{\ssf\lambda}^{(\alpha,\beta)}(x)|^2\! =\ssb|\ssf{\rm
Q}_{\ssf\lambda}^{(\alpha,\beta)}(-x)|^2$ is a decreasing function
of \ssf$x$ \ssf on \ssf$[\ssf0,+\infty)$, according to \eqref{Qplus}
or \eqref{Qminus}. In the second case, consider the continuous and
piecewise differentiable function
\begin{equation*}
M(x)=\ssf\max\, \bigl\{\,\bigl|\ssf{\rm
Q}_{\ssf\lambda}^{(\alpha,\beta)}(x)\bigr|^2, \ssf\bigl|\ssf{\rm
Q}_{\ssf\lambda}^{(\alpha,\beta)}(-x)\bigr|^2\ssf\bigr\}
\end{equation*}
on \ssf$[\ssf0,+\infty)$\ssf. Firstly, if \,$\bigl|\ssf{\rm
Q}_{\ssf\lambda}^{(\alpha,\beta)}(x)\bigr|\ssb
>\ssb\bigl|\ssf{\rm Q}_{\ssf\lambda}^{(\alpha,\beta)}(-x)\bigr|$\ssf,
then
\begin{equation*}\textstyle
\frac\partial{\partial x}\,M(x)\ssf =\ssf \frac\partial{\partial
x}\, \bigl|\ssf{\rm Q}_{\ssf\lambda}^{(\alpha,\beta)}(x)\bigr|^2
\ssb<\ssf0\ssf,
\end{equation*}
according to \eqref{Qplus}. Secondly, if \,$\bigl|\ssf{\rm
Q}_{\ssf\lambda}^{(\alpha,\beta)}(x)\bigr|\ssb <\ssb\bigl|\ssf{\rm
Q}_{\ssf\lambda}^{(\alpha,\beta)}(-x)\bigr|$\ssf, then
\begin{equation*}\textstyle
\frac\partial{\partial x}\,M(x)\ssf =\ssf \frac\partial{\partial
x}\, \bigl|\ssf{\rm Q}_{\ssf\lambda}^{(\alpha,\beta)}(-x)\bigr|^2
\ssb<\ssf0\ssf,
\end{equation*}
according to \eqref{Qminus}. Thirdly, if \,$\bigl|\ssf{\rm
Q}_{\ssf\lambda}^{(\alpha,\beta)}(x)\bigr|\ssb =\ssb\bigl|\ssf{\rm
Q}_{\ssf\lambda}^{(\alpha,\beta)}(-x)\bigr|$ \,for some
\,$x\!>\!0$\ssf, then \ssf $M$ has left and right derivatives at
\ssf$x$\ssf, which are nonpositive, according to \eqref{Qplus} and
\eqref{Qminus}. Thus \ssf$M$ is a decreasing function on
\ssf$[\ssf0,+\infty)$\ssf. In all cases, we conclude in particular
that, for every \ssf$x\!\in\!\R$\ssf,
\begin{equation*}
\bigl|\ssf{\rm Q}_{\ssf\lambda}^{(\alpha,\beta)}(x)\bigr|
\le\bigl|\ssf{\rm Q}_{\ssf\lambda}^{(\alpha,\beta)}(0)\bigr| =1
\quad\text{i.e.}\quad \bigl|\ssf{\rm
G}_\lambda^{(\alpha,\beta)}(x)\bigr| \le{\rm
G}_0^{(\alpha,\beta)}(x)\ssf.
\end{equation*}
\end{proof}

The following results are deduced by  interpolation and duality from
Theorem \ref{KunzeStein} and Proposition \ref{Young}.
\begin{corollary}
\begin{itemize}
\item[(i)]
Let $1\!\le\!p\!<\!q\!\le\!2$\ssf. Then
$$
L^p(\R, A_{\alpha,\beta}(|x|)\ssf dx) \ast_{\alpha,\beta}L^q(\R,
A_{\alpha,\beta}(|x|)\ssf dx) \subset L^q(\R,
A_{\alpha,\beta}(|x|)\ssf dx)\ssf.
$$
\item[(ii)]
Let \,$1\!<\!p\!<\!2$ and \,$\!p\!<\!q\!\le\!\frac p{2-p}$. Then
$$
L^p(\R, A_{\alpha,\beta}(|x|)\ssf dx) \ast_{\alpha,\beta}  L^p(\R,
A_{\alpha,\beta}(|x|)\ssf dx) \subset L^q(\R,
A_{\alpha,\beta}(|x|)\ssf dx)\ssf.
$$
\item[(iii)]
Let \,$2\!<\!p,q\!<\!\infty$ such that \,$\frac
q2\!\le\!p\!<\!q$\ssf. Then
$$
L^p(\R, A_{\alpha,\beta}(|x|)\ssf dx) \ast_{\alpha,\beta}
L^{q^{\prime}}(\R, A_{\alpha,\beta}(|x|)\ssf dx) \subset L^q(\R,
A_{\alpha,\beta}(|x|)\ssf dx)\ssf.
$$
\end{itemize}
\end{corollary}

\section{A special orthogonal system}

In this section we construct an orthogonal basis of \ssf$L^2(\R,
A_{\alpha,\beta}(|x|)\ssf dx)$ \ssf and we compute its
Opdam--Cherednik transform. As limits, we recover the Hermite
functions constructed by Rosenblum \cite{Rosn}.

\begin{proposition}
Let \,$\alpha\!\ge\!\beta\!\ge\!-\frac12$ with
\,$\alpha\!>\!-\frac12$\ssf. For any fixed \,$\delta\!>\!0$\ssf,
consider the sequence of functions
\begin{equation}\label{basis}\begin{cases}
\,{\rm H}_{\ssf2n}^{\,\delta}(x)=(\cosh
x)^{-\alpha-\beta-\delta-2}\,
{\rm P}_n^{(\alpha,\delta)}(1\!-\!2\tanh^2\!x)\ssf,\\
\,{\rm H}_{\ssf2n+1}^{\,\delta}(x)=(\cosh
x)^{-\alpha-\beta-\delta-2}\, {\rm
P}_n^{(\alpha+1,\delta-1)}(1\!-\!2\tanh^2\!x)\,\tanh x\ssf,
\end{cases}\end{equation}
whose definition involves the Jacobi polynomials
\eqref{JacobiPolynomial}. Then \,$\{{\rm H}_{\ssf
n}^{\ssf\delta}\}_{n\in\N}$ is an orthogonal basis of
\,$L^2(\R,A_{\alpha,\beta}(|x|)\ssf dx)$.
\end{proposition}

\begin{proof}
To begin with, let us prove the orthogonality of \ssf$\{{\rm
H}_{\ssf n}^{\ssf\delta}\}_{n\in\N}$ \ssf in
\ssf$L^2(\R,A_{\alpha,\beta}(|x|)\ssf dx)$. Firstly, by oddness
\begin{equation*}
\int_{\ssf\R} {\rm H}_{\ssf2m}^{\ssf\delta}(x)\, {\rm
H}_{\ssf2n+1}^{\ssf\delta}(x)\, A_{\alpha,\beta}(|x|)\, dx =0\ssf.
\end{equation*}
Secondly, we obtain
\begin{align*}
&\int_{\ssf\R} {\rm H}_{\ssf2m}^{\ssf\delta}(x)\, {\rm
H}_{\ssf2n}^{\ssf\delta}(x)\, A_{\alpha,\beta}(|x|)\,
dx\\
&=\,2\ssf\int_{\,0}^{\ssf1}\! {\rm
P}_m^{(\alpha,\delta)}(1\!-\ssb2\ssf y^2)\, {\rm
P}_n^{(\alpha,\delta)}(1\!-\ssb2\ssf y^2)\, y^{\ssf2\alpha+1}\,
(1\!-\ssb y^2)^\delta\,
dy\\
&=\,2^{-\alpha-\delta-1}\ssb \int_{-1}^{\ssf1}\! {\rm
P}_m^{(\alpha,\delta)}(z)\, {\rm P}_n^{(\alpha,\delta)}(z)\,
(1\!-\ssb z)^\alpha\,
(1\!+\ssb z)^\delta\,dz\\
&=\frac{\Gamma(\alpha\!+\!n\!+\!1)\,\Gamma(\delta\!+\!n\!+\!1)}
{(\alpha\!+\!\delta\!+\!2\ssf
n\!+\!1)\;n\ssf!\;\Gamma(\alpha\!+\!\delta\!+\!n\!+\!1)}\;
\delta_{m,n}\,,
\end{align*}
by performing the changes of variables \ssf$y\ssb=\ssb\tanh x$\ssf,
$z\ssb=\!1\!-\ssb y^2$ and by using the orthogonality of Jacobi
polynomials  (see for instance \cite{AAR})\,:
\begin{align*}
&\int_{-1}^{+1}\hspace{-1mm} {\rm P}_m^{(\alpha,\delta)}(z)\,{\rm
P}_n^{(\alpha,\delta)}(z)\,
(1\!-\!z)^\alpha\,(1\!+\!z)^\delta\,dz\\
&=\,2^{\ssf\alpha+\delta+1}\,
\frac{\Gamma(\alpha\!+\!n\!+\!1)\,\Gamma(\delta\!+\!n\!+\!1)}
{(\alpha\!+\!\delta\!+\!2\ssf
n\!+\!1)\;n\ssf!\;\Gamma(\alpha\!+\!\delta\!+\!n\!+\!1)}\;
\delta_{m,n}\ssf.
\end{align*}
Thirdly, by the same arguments
\begin{align*}
&\int_{\ssf\R} {\rm H}_{\ssf2m+1}^{\ssf\delta}(x)\, {\rm
H}_{\ssf2n+1}^{\ssf\delta}(x)\,
A_{\alpha,\beta}(|x|)\,dx\\
&=\,2\ssf\int_{\,0}^{\ssf1}\! {\rm
P}_m^{(\alpha+1,\delta-1)}(1\!-\ssb2\ssf y^2)\, {\rm
P}_n^{(\alpha+1,\delta-1)}(1\!-\ssb2\ssf y^2)\,
y^{\ssf2\alpha+3}\,(1\!-\ssb y^2)^{\delta-1}\,dy\\
&=\,2^{-\alpha-\delta-1}\ssb\int_{-1}^{\ssf1}\! {\rm
P}_m^{(\alpha+1,\delta-1)}(z)\,{\rm P}_n^{(\alpha+1,\delta-1)}(z)\,
(1\!-\ssb z)^{\alpha+1}\,(1\!+\ssb z)^{\delta-1}\,dz\\
&=\frac{\Gamma(\alpha\!+\!n\!+\!2)\,\Gamma(\delta\!+\!n)}
{(\alpha\!+\!\delta\!+\!2n\!+\!1)\,n\ssf!\,\Gamma(\alpha\!+\!\delta\!+\!n\!+\!1)}\;
\delta_{m,n}\ssf.
\end{align*}
Let us turn to the completeness of \ssf$\{{\rm H}_{\ssf
n}^{\ssf\delta}\}_{n\in\N}$ \ssf in
\ssf$L^2(\R,A_{\alpha,\beta}(|x|)\ssf dx)$\ssf. Recall (see for
instance \cite{AAR}) that the Jacobi polynomials \ssf$\{{\rm
P}_n^{(\tilde\alpha,\tilde\delta)}\}_{n\in\N}$ \ssf span a dense
subspace of \ssf$L^2(\,]\!-\!1,1\ssf[\,,
(1\!-\!z)^{\tilde\alpha}\ssf(1\!+\!z)^{\tilde\delta}\ssf dz)$\ssf.
By the above changes of variables, we deduce that \ssf$\{{\rm
H}_{\ssf2n}^{\ssf\delta}\}_{n\in\N}$ \ssf and \ssf$\{{\rm
H}_{\ssf2n+1}^{\ssf\delta}\}_{n\in\N}$ \ssf span dense subspaces of
\ssf$L^2(\R,A_{\alpha,\beta}(|x|)\ssf dx)_{\text{\rm e}}$ and
\ssf$L^2(\R,A_{\alpha,\beta}(|x|)\ssf dx)_{\text{\rm o}}$
respectively.
\end{proof}

\begin{remark}
In \eqref{basis}, let us replace \,$\delta$ by \,$\varepsilon^{-2}$,
$x$ by \,$\varepsilon\ssf x$ and let
\,$\varepsilon\!\searrow\!0$\ssf. As
\begin{equation*}
(\cosh\varepsilon x)^{-\alpha-\beta-\ssf\varepsilon^{-2}-2}
\longrightarrow\;e^{-\frac{x^2}2}
\end{equation*}
and
\begin{equation*}\textstyle
{\rm P}_{\ssf n}^{(a,\ssf b\ssf+\ssf\varepsilon^{-2})}
(1\!-\ssb2\tanh^2\ssb\varepsilon\ssf x)
=\frac{(a\ssf+1)_n}{n\ssf!}\,{}_2{\rm F}_{\!1}(\ssb-\ssf n\ssf,
a\ssb+\ssb b\ssb+\ssb\varepsilon^{-2}\!+\ssb n\ssb+\!1\ssf;
a\ssb+\!1\ssf;\tanh^2\ssb\varepsilon\ssf x)
\end{equation*}
tends to the Laguerre polynomial
\begin{equation*}\textstyle
{\rm L}_{\ssf n}^{\ssf a}(x^2) =\frac{(a\ssf+1)_n}{n\ssf!}\,
{}_1{\rm F}_{\!1}(\ssb-\ssf n\,;a\ssb+\!1\ssf;x^2\ssf)\ssf,
\end{equation*}
we recover in the limit the even and odd Hermite functions
constructed by Rosenblum \cite[Definition 3.4]{Rosn} in the rational
Dunkl  setting\,{\rm:}
\begin{equation*}\begin{cases}
\;{\rm H}_{\ssf2n}^{\,\varepsilon^{-2}}(\varepsilon x)
\longrightarrow\,
e^{-\frac{x^2}2}\,{\rm L}_{\ssf n}^{\ssf\alpha}(x^2)\,,\\
\;\varepsilon^{-1}\, {\rm
H}_{\ssf2n+1}^{\,\varepsilon^{-2}}(\varepsilon x) \longrightarrow\,
e^{-\frac{x^2}2}\,{\rm L}_{\,n}^{\alpha+1}(x^2)\,x\,.
\end{cases}\end{equation*}
\end{remark}

\begin{theorem}\label{FH}
The Opdam--Cherednik transform of \,$\{H_n^\delta\}_{n\in\N}$ is
given by
\begin{equation}\begin{aligned}\label{FHeven}
\mathcal{F}({\rm H}_{\ssf2n}^{\ssf\delta})(\lambda)
&=\frac{(-1)^n}{n\ssf!}\frac{\Gamma(\alpha\!+\!1)\,
\Gamma(\frac{\delta+1+i\lambda}2)\,\Gamma(\frac{\delta+1-i\lambda}2)}
{\Gamma(\frac{\alpha+\beta+\delta}2\!+\!n\!+\!1)\,
\Gamma(\frac{\alpha-\beta+\delta}2\!+\!n\!+\!1)}\\
&\ssf\times\,{\rm P}_n\Bigl(-\ssf\frac{\lambda^2}4;
\frac{\delta\!+\!1}2,\frac{\delta\!+\!1}2,
\frac{\alpha\!+\!\beta\!+\!1}2,\frac{\alpha\!-\!\beta\!+\!1}2\Bigr)\,,
\end{aligned}\end{equation}
and
\begin{equation}\begin{aligned}\label{FHodd}
\mathcal{F}({\rm H}_{\ssf2n+1}^{\ssf\delta})(\lambda)
&=\frac{(-1)^n}{n\ssf!}\,
\frac{(\rho\!+\!i\lambda)\,\Gamma(\alpha\!+\!1)\,
\Gamma(\frac{\delta+1+i\lambda}2)\,
\Gamma(\frac{\delta+1-i\lambda}2)}
{2\,\Gamma(\frac{\alpha+\beta+\delta}2\!+\!n\!+\!2)\,
\Gamma(\frac{\alpha-\beta+\delta}2\!+\!n\!+\!1)}\\
&\ssf\times{\rm P}_n\Bigl(-\ssf\frac{\lambda^2}4;
\frac{\delta\!+\!1}2,\frac{\delta-1}2,
\frac{\alpha\!+\!\beta\!+\!3}2,\frac{\alpha\!-\!\beta\!+\!1}2\Bigr)
\end{aligned}\end{equation}
where
\begin{align*}
{\rm P}_n(t^2;a\ssf,b\ssf,c\ssf,d\ssf)
&=(a\!+\!b)_n\,(a\!+\!c)_n\,(a\!+\!d)_n\\
&\times\,{}_4{\rm F}_{\!3} \Bigr(\begin{matrix} \ssf-\ssf n\ssf,\ssf
a\!+\!b\!+\!c\!+\!d\!+\!n\!-\!1\ssf,
\ssf a\!+\!t\ssf,\ssf a\!-\!t\,\\
a\!+\!b\ssf,\ssf a\!+\!c\ssf,\ssf a\!+\!d
\end{matrix};1\Bigr)
\end{align*}
denotes the Wilson polynomials.
\end{theorem}

\begin{proof}
By evenness, the Opdam--Cherednik transform $\mathcal{F}({\rm
H}_{\ssf2n}^{\ssf\delta})$ coincides with the Jacobi transform
$\mathcal F_{\alpha, \beta}({\rm H}_{\ssf2n}^{\ssf\delta})$. Thus
\eqref{FHeven} amounts to Formula (9.4) in \cite{K2}. Let us recall
its proof, which was sketched in \cite[Section 9]{K2} and which will
be used for \eqref{FHodd}. On one hand, we expand
\begin{align}\label{Hevencosh}\textstyle
{\rm H}_{\ssf2n}^{\ssf\delta}(x) &=(-1)^n\,(\cosh
x)^{-\alpha-\beta-\delta-2}\, {\rm P}_{\ssf
n}^{(\delta,\alpha)}(2\tanh^2\!x\ssb-\!1)
\nonumber\\
&=(-1)^n\,\frac{(\delta\!+\!1)_n}{n\ssf!}\, (\cosh
x)^{-\alpha-\beta-\delta-2}\, {}_2{\rm
F}_{\!1}(-n\ssf,\alpha\!+\!\delta\!+\!n\!+\!1;
\delta\!+\!1;\cosh^{-2}\!x)
\nonumber\\
&=\frac{(-1)^n}{n\ssf!}\,(\delta\!+\!1)_n\,\sum_{m=0}^n\,
\frac{(-\ssf
n)_m\,(\alpha\!+\!\delta\!+\!n\!+\!1)_m}{(\delta\!+\!1)_m\,m\ssf!}\,
(\cosh x)^{-\alpha-\beta-\delta-2\ssf m-2}
\end{align}
in negative powers of \ssf$\cosh x$\ssf, using the symmetry
\begin{equation*}
{\rm P}_{\ssf n}^{(\alpha,\delta)}(x) =(-1)^n\,{\rm P}_{\ssf
n}^{(\delta,\alpha)}(-x)
\end{equation*}
and the definition \eqref{JacobiPolynomial} of Jacobi polynomials.
On the other hand, recall the following Jacobi transform
\cite[Formula (9.1)]{K2}\,:
\begin{equation}\label{Jcosh}
\int_{\,0}^{+\infty}\! (\cosh x)^{-\alpha-\beta-\mu-1}\,
\varphi_\lambda^{(\alpha,\beta)}(x)\, A_{\alpha,\beta}(x)\,dx
=\frac{\Gamma(\alpha\!+\!1)\,
\Gamma(\frac{\mu+i\lambda}2)\,\Gamma(\frac{\mu-i\lambda}2)}
{2\,\Gamma(\frac{\alpha+\beta+\mu+1}2)\,\Gamma(\frac{\alpha-\beta+\mu+1}2)}\,.
\end{equation}
We conclude by combining \eqref{Hevencosh} and \eqref{Jcosh}\,:

\begin{align*}
\mathcal{F}({\rm H}_{\ssf2n}^{\ssf\delta})(\lambda)
&=\int_{\ssf\R}{\rm H}_{\ssf2n}^{\ssf\delta}(x)\,
\varphi_\lambda^{(\alpha,\beta)}(x)\,
A_{\alpha,\beta}(|x|)\,dx\\
&=\frac{(-1)^n}{n\ssf!}\,(\delta\!+\!1)_n\,\sum_{m=0}^n\,
\frac{(-\ssf n)_m\,(\alpha\!+\!\delta\!+\!n\!+\!1)_m}{(\delta\!+\!1)_m\,m\ssf!}\\
&\ssf\times\ssf\underbrace{2\int_{\,0}^{+\infty}\! (\cosh
x)^{-\alpha-\beta-\delta-2\ssf m-2}\,
\varphi_\lambda^{(\alpha,\beta)}(x)\, A_{\alpha,\beta}(x)\,dx
}_{\displaystyle\frac{\Gamma(\alpha\!+\!1)\,
\Gamma(\frac{\delta+1+i\lambda}2\!+\!m)\,
\Gamma(\frac{\delta+1-i\lambda}2\!+\!m)}
{\Gamma(\frac{\alpha+\beta+\delta}2\!+\!m\!+\!1)\,
\Gamma(\frac{\alpha-\beta+\delta}2\!+\!m\!+\!1)}}\\
&=\frac{(-1)^n}{n\ssf!}\,\frac{\Gamma(\alpha\!+\!1)\,
\Gamma(\frac{\delta+1+i\lambda}2)\,
\Gamma(\frac{\delta+1-i\lambda}2)}
{\Gamma(\frac{\alpha+\beta+\delta}2\!+\!1)\,
\Gamma(\frac{\alpha-\beta+\delta}2\!+\!1)}\,
(\delta\!+\!1)_n\\
&\ssf\times\underbrace{\sum_{m=0}^n\, \frac{(-\ssf n)_m\,
(\alpha\!+\!\delta\!+\!n\!+\!1)_m\, (\frac{\delta+1+i\lambda}2)_m\,
(\frac{\delta+1-i\lambda}2)_m} {(\delta\!+\!1)_m\,
(\frac{\alpha+\beta+\delta}2\!+\!1)_m\,
(\frac{\alpha-\beta+\delta}2\!+\!1)_m\, m\ssf!}
}_{\displaystyle{}_4{\rm F}_{\!3} \Bigl(\begin{matrix} \ssf-\ssf
n\ssf,\ssf\alpha\!+\!\delta\!+\!n\!+\!1\ssf,
\frac{\delta+1+i\lambda}2\ssf,\frac{\delta+1-i\lambda}2\,\\
\delta\!+\!1\ssf, \frac{\alpha+\beta+\delta}2\!+\!1\ssf,
\frac{\alpha-\beta+\delta}2\!+\!1
\end{matrix};1\Bigr)}\\
&=\frac{(-1)^n}{n\ssf!}\,\frac{\Gamma(\alpha\!+\!1)\,
\Gamma(\frac{\delta+1+i\lambda}2)\,
\Gamma(\frac{\delta+1-i\lambda}2)}
{\Gamma(\frac{\alpha+\beta+\delta}2\!+\!n\!+\!1)\,
\Gamma(\frac{\alpha-\beta+\delta}2\!+\!n\!+\!1)}\\
&\ssf\times{\rm P}_n\Bigl(-\ssf\frac{\lambda^2}4;
\frac{\delta\!+\!1}2,\frac{\delta\!+\!1}2,
\frac{\alpha\!+\!\beta\!+\!1}2,\frac{\alpha\!-\!\beta\!+\!1}2\Bigr)
\end{align*}
Similarly,
\begin{equation*}\begin{aligned}\label{Hoddcosh}\textstyle
{\rm H}_{\ssf2n+1}^{\ssf\delta}(x) &=(-1)^n\,(\cosh
x)^{-\alpha-\beta-\delta-2}\, {\rm P}_{\ssf
n}^{(\delta-1,\alpha+1)}(2\tanh^2\!x\ssb-\!1)\, \tanh x
\nonumber\\
&=\frac{(-1)^n}{n\ssf!}\,(\delta)_n\,(\sinh x)\, (\cosh
x)^{-\alpha-\beta-\delta-3}\, {}_2{\rm
F}_{\!1}(-n\ssf,\alpha\!+\!\delta\!+\!n\!+\!1;
\delta\ssf;\cosh^{-2}\!x)
\nonumber\\
&=\frac{(-1)^n}{n\ssf!}\,(\delta)_n\,(\sinh x)\,\sum_{m=0}^n\,
\frac{(-\ssf n)_m\,(\alpha\!+\!\delta\!+\!n\!+\!1)_m}
{(\delta)_m\,m\ssf!}\, (\cosh x)^{-\alpha-\beta-\delta-2\ssf m-3}
\end{aligned}\end{equation*}
and
\begin{align*}
\mathcal{F}({\rm H}_{\ssf2n+1}^{\ssf\delta})(\lambda)
&=\int_{\ssf\R}{\rm H}_{\ssf2n+1}^{\ssf\delta}(x)\, {\rm
G}_{\lambda,{\rm o}}^{(\alpha,\beta)}(x)\,
A_{\alpha,\beta}(|x|)\,dx\\
&=\ssf\frac{\rho\!+\!i\lambda}{4\ssf(\alpha\!+\!1)}
\int_{\ssf\R}{\rm H}_{\ssf2n+1}^{\ssf\delta}(x)\,
\varphi_\lambda^{(\alpha+1,\beta+1)}(x)\,
(\sinh2x)\,A_{\alpha,\beta}(|x|)\,dx\\
&=\ssf\frac{\rho\!+\!i\lambda}{\alpha\!+\!1}
\int_{\,0}^{+\infty}\hspace{-1mm}(\sinh x\cosh x)^{-1}\, {\rm
H}_{\ssf2n+1}^{\ssf\delta}(x)\,
\varphi_\lambda^{(\alpha+1,\beta+1)}(x)\,
A_{\alpha+1,\beta+1}(x)\,dx\\
\end{align*}

\begin{align*}
&=\frac{(-1)^n}{n\ssf!}\,\frac{\rho\!+\!i\lambda}{\alpha\!+\!1}\,
(\delta)_n\,\sum_{m=0}^n\, \frac{(-\ssf
n)_m\,(\alpha\!+\!\delta\!+\!n\!+\!1)_m}
{(\delta)_m\,m\ssf!}\\
&\ssf\times\underbrace{\int_{\,0}^{+\infty}\! (\cosh
x)^{-\alpha-\beta-\delta-2\ssf m-4}\,
\varphi_\lambda^{(\alpha+1,\beta+1)}(x)\,
A_{\alpha+1,\beta+1}(x)\,dx
}_{\displaystyle\frac{\Gamma(\alpha\!+\!2)\,
\Gamma(\frac{\delta+1+i\lambda}2\!+\!m)\,
\Gamma(\frac{\delta+1-i\lambda}2\!+\!m)}
{2\,\Gamma(\frac{\alpha+\beta+\delta}2\!+\!m\!+\!2)\,
\Gamma(\frac{\alpha-\beta+\delta}2\!+\!m\!+\!1)}}\\
&=\frac{(-1)^n}{n\ssf!}\,
\frac{(\rho\!+\!i\lambda)\,\Gamma(\alpha\!+\!1)\,
\Gamma(\frac{\delta+1+i\lambda}2)\,
\Gamma(\frac{\delta+1-i\lambda}2)}
{2\,\Gamma(\frac{\alpha+\beta+\delta}2\!+\!2)\,
\Gamma(\frac{\alpha-\beta+\delta}2\!+\!1)}\,
(\delta)_n\\
&\ssf\times\underbrace{\sum_{m=0}^n\, \frac{(-\ssf n)_m\,
(\alpha\!+\!\delta\!+\!n\!+\!1)_m\, (\frac{\delta+1+i\lambda}2)_m\,
(\frac{\delta+1-i\lambda}2)_m} {(\delta)_m\,
(\frac{\alpha+\beta+\delta}2\!+\!2)_m\,
(\frac{\alpha-\beta+\delta}2\!+\!1)_m\, m\ssf!}
}_{\displaystyle{}_4{\rm F}_{\!3} \Bigl(\begin{matrix} \ssf-\ssf
n\ssf,\ssf\alpha\!+\!\delta\!+\!n\!+\!1\ssf,
\frac{\delta+1+i\lambda}2\ssf,\frac{\delta+1-i\lambda}2\,\\
\delta\ssf, \frac{\alpha+\beta+\delta}2\!+\!2\ssf,
\frac{\alpha-\beta+\delta}2\!+\!1
\end{matrix};1\Bigr)}\\
&=\frac{(-1)^n}{n\ssf!}\,
\frac{(\rho\!+\!i\lambda)\,\Gamma(\alpha\!+\!1)\,
\Gamma(\frac{\delta+1+i\lambda}2)\,
\Gamma(\frac{\delta+1-i\lambda}2)}
{2\,\Gamma(\frac{\alpha+\beta+\delta}2\!+\!n\!+\!2)\,
\Gamma(\frac{\alpha-\beta+\delta}2\!+\!n\!+\!1)}\\
&\ssf\times{\rm P}_n\Bigl(-\ssf\frac{\lambda^2}4;
\frac{\delta\!+\!1}2,\frac{\delta-1}2,
\frac{\alpha\!+\!\beta\!+\!3}2,\frac{\alpha\!-\!\beta\!+\!1}2\Bigr)
\end{align*}
\end{proof}

By comparing the Opdam--Cherednik transform of \ssf${\rm H}_{\ssf
n}^{\ssf\delta}$ with the particular case \ssf${\rm H}_{\ssf
0}^{\ssf\delta}(x)\ssb=\ssb(\cosh x)^{-\alpha-\beta-\delta-2}$, we
obtain the following Rodrigues type formula.

\begin{corollary}
Consider the polynomials
\begin{equation*}\begin{cases}
\,\tilde{\rm P}_{2n}^{\,\delta}(t) =\frac{\vphantom{\big|}(-1)^n}
{\vphantom{\big|}n\ssf!\, (\frac{\alpha+\beta+\delta}2+1)_n\,
(\frac{\alpha-\beta+\delta}2+1)_n}\, {\rm
P}_n\bigl(\frac{t^2}4;\frac{\delta+1}2,\frac{\delta+1}2,
\frac{\alpha+\beta+1}2,\frac{\alpha-\beta+1}2\bigr)\ssf,\\
\,\tilde{\rm P}_{2n+1}^{\,\delta}(t)
=\frac{\vphantom{\big|}(-1)^n\,(\rho\ssf+\ssf t)}
{\vphantom{\big|}2\,n\ssf!\, (\frac{\alpha+\beta+\delta}2+1)_{n+1}\,
(\frac{\alpha-\beta+\delta}2+1)_n}\; {\rm
P}_n\bigl(\frac{t^2}4;\frac{\delta+1}2,\frac{\delta-1}2,
\frac{\alpha+\beta+3}2,\frac{\alpha-\beta+1}2\bigr)\ssf.
\end{cases}\end{equation*}
Then
\begin{equation}\label{Rodrigues}
\,{\rm H}_{\ssf n}^{\ssf\delta}(x) =\ssf\tilde{\rm
P}_n^{\ssf\delta}({\rm T}_{\vphantom{0}\ssf x}^{(\alpha,\beta)})\ssf
(\cosh x)^{-\alpha-\beta-\delta-2} \qquad\forall\,n\!\in\!\N\ssf.
\end{equation}
In other words, by replacing in the expansion of the polynomial
\,$\tilde{\rm P}_n^{\ssf\delta}(t)$ the variable~\,$t$ by the
Dunkl--Cherednik operator \ssf$T^{(\alpha,\beta)}$\ssb, one obtains
a differential--difference operator, whose action on the function
\,$(\cosh x)^{-\alpha-\beta-\delta-2}$ yields the function ${\rm
H}_{\ssf n}^{\ssf\delta}(x)\ssf$.
\end{corollary}


\end{document}